\documentclass[12pt]{article}

\usepackage{a4wide}
\usepackage{amsmath}
\usepackage{amsthm}
\usepackage{amssymb}
\usepackage[english]{babel}
\usepackage{graphicx}

\newcommand{\R}{\mathbb{R}}

\newcommand{\Z}{\mathbb{Z}}
\newcommand{\Sd}{\textup{Sd}} 
\newcommand{\ord}{\textup{Ord}} 
\newcommand{\for}{\textup{For}} 
\newcommand{\lk}{\textup{Lk}} 

\DeclareMathAlphabet{\mathcalligra}{T1}{calligra}{m}{n}

\DeclareFontFamily{U}{mathx}{\hyphenchar\font45}
\DeclareFontShape{U}{mathx}{m}{n}{
      <5> <6> <7> <8> <9> <10>
      <10.95> <12> <14.4> <17.28> <20.74> <24.88>
      mathx10
      }{}
\DeclareSymbolFont{mathx}{U}{mathx}{m}{n}
\DeclareFontSubstitution{U}{mathx}{m}{n}
\DeclareMathAccent{\widecheck}{0}{mathx}{"71}

\newtheorem{theo}{Theorem}[section]
\newtheorem{lemma}[theo]{Lemma}
\newtheorem{prop}[theo]{Proposition}
\newtheorem{cor}[theo]{Corollary}
\newtheorem{remark}[theo]{Remark}
\newtheorem{ex}[theo]{Example}

\newtheorem{defi}[theo]{Definition}

\begin{document}
\title{Morse shellings on products}
\author{Jean-Yves Welschinger}
\maketitle
\begin{abstract} 
\vspace{0.5cm}

We recently defined a property of Morse shellability (and tileability) of finite simplicial complexes which extends the classical one and its relations with discrete Morse theory.
We now prove that the product of two Morse tileable or shellable simplicial complexes carries Morse tileable or shellable triangulations under some tameness condition, and that any tiling or shelling becomes tame after one barycentric subdivision. 
We deduce that any finite product of closed manifolds of dimensions less than four carries Morse shellable triangulations whose critical and $h$-vectors are palindromic. 
We also prove that the $h$-vector of a Morse tiling is always palindromic in dimension less than four or in the case of an $h$-tiling, provided its critical vector is palindromic. 


\end{abstract}

\section{Introduction}

Recall that the face vector of a finite $n$-dimensional simplicial complex $K$ encodes the number of simplices it contains in each dimension, that is its number of vertices, edges and so on, see \S \ref{subsecsimplcx}. When $K$ is the boundary of a convex polytope for example, it has been (strikingly) understood by L.J. Billera, C.W. Lee and R. P. Stanley what this face vector can be, confirming an earlier conjecture of P. McMullen, see \cite{BilLee, Stan, Fult, Z}. The answer is expressed in terms of its $h$-vector $h(K) = (h_0 (K), \dots , h_{n+1} (K))$, which is a linear recombination of its face vector, see \S \ref{subsecsimplcx}, and turns out to coincide with the list of Betti numbers of the toric variety $X(K)$ associated to the convex polytope. In particular, Poincaré duality in $X(K)$ implies that $h_j (K) = h_{n+1-j} (K)$ for every $j \in \{0, \dots , n+1\}$, a result which also follows directly from the Dehn-Sommerville relations, see \cite{McMu, Mac, Klain, Z} or also Theorem $1.1$ of \cite{SaWelDCG}. We will declare such a vector to be palindromic, see Definition \ref{defpalind}. In general, it is unclear how to understand $h$-vectors, except at least for shellable complexes, see \S $8.3$ of \cite{Z}. We introduced in \cite{SaWel1} a notion of tiling of simplicial complexes and when such a tiling $\tau$ exists, e.g.  for shellable complexes, defined its $h$-vector $h(\tau) = (h_0 (\tau), \dots , h_{n+1} (\tau))$ to be the number of tiles of each type used by $\tau$, see \S \ref{subsecMorseshel}. A tile here, or basic tile, is just a simplex deprived of several of its codimension one faces, whose number is called its order, so that closed and open simplices are particular ones, of minimal and maximal order respectively, see Definition \ref{defbasictile}. By Theorem $4.9$ and Corollary $4.10$ of \cite{SaWel1}, two $h$-tilings $\tau$ and $\tau'$ of a simplicial complex $K$ have same $h$-vector provided $h_0 (\tau) = h_0 (\tau')$ and when moreover the latter equals one, this $h$-vector coincides with the $h$-vector of $K$. These results thus sometimes provide a geometric meaning of $h$-vectors but also provide a larger class of vectors of interest, for $h_0 (\tau)$ need not be one. If among the closed manifolds, only the spheres carry shellable triangulations \cite{Koz, Whi}, we do not know which ones carry $h$-tileable triangulations. We however prove the following, see Corollary \ref{cortores}. 

\begin{theo}
\label{theointro1}
The product of a sphere and a torus of any dimensions carries $h$-tileable triangulations. 
\end{theo}
Definitely, a closed manifold of even dimension and negative Euler characteristic has no $h$-tileable triangulations though, see Lemma \ref{lemmaEuler}. 

Influenced by the discrete Morse theory of R. Forman \cite{For1, For2}, we enlarged in \cite{SaWel2} the collection of tiles in each dimension by allowing a unique face of higher codimension to be removed from a simplex, thus introducing Morse tiles, see Definition \ref{defMorsetile}. These include a collection of critical tiles of any index and led to properties of Morse tileable and shellable complexes, see \S \ref{subsecMorseshel}. Among basic tiles, closed and open simplices are the only critical ones, of minimal and maximal indices respectively. We proved that any triangulation on a closed surface is Morse shellable and that any closed three-manifold carries Morse shellable triangulations, see Theorems $1.3$ and $1.4$ of \cite{SaWel2}. Moreover, Morse tilings carry compatible discrete vector fields and in the case of Morse shellings, these are gradient vector fields of discrete Morse functions whose critical points are in one-to-one correspondence with the critical tiles of the tiling, preserving the index, see Theorem $1.2$ of \cite{SaWel2}. We now prove, see Corollary \ref{corproductmanifolds}.
\begin{theo}
\label{theointro2}
Any finite product of closed manifolds of dimensions less than four carries Morse shellable triangulations. 
\end{theo}

Recall that H. Bruggesser and P. Mani proved that the boundary of every convex polytope is shellable, while some triangulations on spheres are not, see \cite{BruMan, Lick, Z}. We do not know any closed triangulated manifold which is not Morse tileable or shellable.
We encode the number of critical tiles of each index used by an $n$-dimensional Morse tiling $\tau$ in a critical vector $c(\tau) = (c_0 (\tau), \dots , c_{n} (\tau))$ 
and likewise, the number of tiles of each order it uses in an $h$-vector $h(\tau) = (h_0 (\tau), \dots , h_{n+1} (\tau))$, as in the case of $h$-tilings, see \S \ref{subsecMorseshel}. We then prove, see Corollary \ref{corpalindromic} and Theorem \ref{theomvector}.

\begin{theo}
\label{theointro3}
The $h$-vector of an $h$-tiling on a closed triangulated manifold is palindromic iff its $c$-vector is. Likewise, the $h$-vector of a Morse tiling on a closed triangulated manifold of dimension less than four is palindromic iff its $c$-vector is. 
\end{theo}
The tilings given by Theorems \ref{theointro1} and \ref{theointro2} can be chosen to have palindromic $h$-vectors as well. 
Does there exist a Morse tiling on a closed triangulated manifold which has palindromic $c$-vector but non palindromic $h$-vector? It would also be of interest to get Theorem \ref{theointro2} in any dimension. In dimension three, it has been obtained in \cite{SaWel2} by successive attachments of triangulated handles equipped with Morse shellings. We now prove the existence of such shellings on every handle, see Corollary \ref{corhandle}.
\begin{theo}
\label{theointro4}
For every $0 \leq k \leq n$, the handle $\stackrel{\circ}{\Delta}_k \times \Delta_{n-k}$ carries Morse shellable triangulations using a unique critical tile, of index $k$.
\end{theo}
The core of the paper actually aims at proving that the product of two Morse tileable or shellable simplicial complexes carries Morse tileable or shellable triangulations. We first prove this result for single Morse tiles, see Theorem \ref{theoproducttiles} of which Theorem \ref{theointro4} is a special case, and then observe a duality phenomenon which makes it possible to get the palindromic property, see Theorem \ref{theoduality}. We then prove the result in general under some tameness condition on the tilings, see \S \ref{subsectame}, to get.
\begin{theo}
\label{theointro5}
Let $K_1$ and $K_2$ be finite simplicial complexes equipped with tame Morse tilings (resp. shellings) $\tau_1$ and $\tau_2$. Then, $K_1 \times K_2$ carries tame Morse tileable (resp. shellable) primitive triangulations. Moreover, if $\tau_1$, $\tau_2$ are pure dimensional, these Morse tilings have palindromic $h$-vectors provided $h(\tau_1)$ and $h(\tau_2)$ are palindromic.
\end{theo}
The critical vector of such tilings on $K_1 \times K_2$ is a product of the ones of $\tau_1$ and $\tau_2$ while $\tau_1$, $\tau_2$ are always pure dimensional in the case of triangulated manifolds, see Lemma \ref{lemmapure} and Theorem \ref{theoproduct}. Theorem \ref{theointro5} suffices to deduce Theorems \ref{theointro1} and \ref{theointro2}, for the tilings or shellings on each factor can be chosen to be tame and it has a counterpart which produces $h$-tilings as well, see Theorem \ref{theoproducthtiling}. 
In fact, any Morse tiling or shelling becomes tame after a single barycentric subdivision, see Proposition \ref{propbarycsubdtame}.
We finally provide many examples of Morse shellings throughout the paper, see in particular \S \ref{subsecexamples}.

We recall in section \ref{secprelim} the classical notions of face and $h$-vectors of simplicial complexes, the notions of tilings and shellings defined in \cite{SaWel1, SaWel2} and we introduce the tameness condition needed to get Theorem \ref{theointro5}. We then formulate our main results in section \ref{secmain}, devoting \S \ref{subsecpalindrome} to the palindromic property and Theorem \ref{theointro3}, \S \ref{subsecproducts} to Theorem \ref{theointro5} and \S \ref{subsechandles} to the special case of single tiles and Theorem \ref{theointro4}. We study in \S \ref{sectriangulations} the cartesian products of two simplices together with the shellings of its staircase triangulations. This makes it possible to prove the main results in \S \S \ref{secshellings} and \ref{sectilings}.\\

\textbf{Acknowledgement:}
This work was partially supported by the ANR project MICROLOCAL (ANR-15CE40-0007-01).

\tableofcontents

\section{Preliminaries}
\label{secprelim}

\subsection{Simplicial complexes}
\label{subsecsimplcx}

Let $n$ be a non-negative integer. An $n$-simplex is the convex hull of $n+1$ points affinely independent in some real affine space. A face of a simplex is the convex hull of a subset of its vertices and we call it a facet when it has codimension one in the simplex. The standard $n$-simplex is the convex hull of the standard basis of $\R^{n+1}$. It will be denoted by $\Delta_{[n]}$, or sometimes just by $\Delta_n$, fixing an identification between its vertices and the set of integers $[n] = \{ 0, \dots, n \}$. Likewise, for every subset $J$ of $\{0, \dots , n\}$, we will denote by $\Delta_J$ the face of $\Delta_{[n]}$ whose vertices belong to $J$. A total order on the vertices of any simplex prescribes then an affine isomorphism with the standard simplex of the corresponding dimension.

A finite simplicial complex $K$ is a finite collection of simplices which contains all faces of its simplices and such that the intersection of any two simplices in this collection is a face of each of them, see~\cite{Munk, EilSteen}. The dimension of such a complex is the maximal dimension of its simplices and it is said to be pure $n$-dimensional if all the simplices that are maximal with respect to the inclusion are of dimension $n$. Such a simplicial complex $K$ inherits a topology and the underlying topological space is usually denoted by $\vert K \vert$, see \cite{Munk, EilSteen}. When it gets homeomorphic to some manifold, any such homeomorphism is called a triangulation of the manifold. 

The face vector or $f$-vector of an $n$-dimensional finite simplicial complex $K$ is the vector $f(K) = (f_{-1} (K), f_{0} (K), \dots , f_n (K))$, where for every $j \in \{0, \dots , n\}$, $f_j (K)$ denotes the number of $j$-simplices of $K$ while $f_{-1} (K) = 1$ counts the empty set. Likewise, the $h$-vector $h(K) = (h_{0} (K), \dots , h_{n+1} (K))$ of $K$ is defined by the relation
$\sum_{i=0}^{n+1} h_i (K) X^{n+1-i} = \sum_{i=0}^{n+1} f_{i-1} (K) (X-1)^{n+1-i} $, see \cite{McMu, Stan1, Fult, Z}. 

\begin{ex}
\label{exhvect}
The boundary of a simplex is homeorphic to a sphere. Its $h$-vector equals $(1, \dots , 1)$.
\end{ex}

Let us finally recall that a finite simplicial complex is said to be shellable iff there exists an order $\sigma_1, \dots , \sigma_N$ of its maximal simplices such that for every $i \in \{2, \dots, N \}$, $\sigma_i \cap \big( \cup_{j=1}^{N-1} \sigma_j \big)$ is non-empty of pure dimension $\dim \sigma_i -1$, see \cite{Koz, Z} for instance. This means that the simplices $\sigma_1, \dots , \sigma_N$ are not proper faces of any other simplex in $K$ and that any simplex in $\sigma_i \cap \big( \cup_{j=1}^{N-1} \sigma_j \big)$ is a face of a $(\dim \sigma_i -1)$-dimensional one in this intersection. It is convenient for us to allow this intersection for being empty, so that a shelling for us need not be connected, see Remark $2.16$ of \cite{SaWel2} and \S \ref{subsecMorseshel}.

\subsection{Morse shellings}
\label{subsecMorseshel}

We now recall the notions of tilings and shellings introduced in \cite{SaWel1, SaWel2}.

\begin{defi}
\label{defbasictile}
A basic tile of dimension $n$ and order $k \in \{0, \dots , n+1\}$ is an $n$-simplex deprived of $k$ of its facets. 
\end{defi}

Two basic tiles of same dimension and order are isomorphic to each other via some affine isomorphism. We denote by $T^n_k = \Delta_{[n]} \setminus \cup_{j=0}^{k-1} \Delta_{[n] \setminus \{j \}}$ the standard basic tile of dimension $n$ and order $k$, compare \cite{SaWel1}.
\begin{ex}
1) The open (resp. closed) $n$-simplex is the basic tile of dimension $n$ and order $n+1$ (resp. $0$).

2) Figure \ref{fig2tiles} depicts the four isomorphism classes of basic tiles in dimension two.
 \begin{figure}[h]
   \begin{center}
    \includegraphics[scale=1]{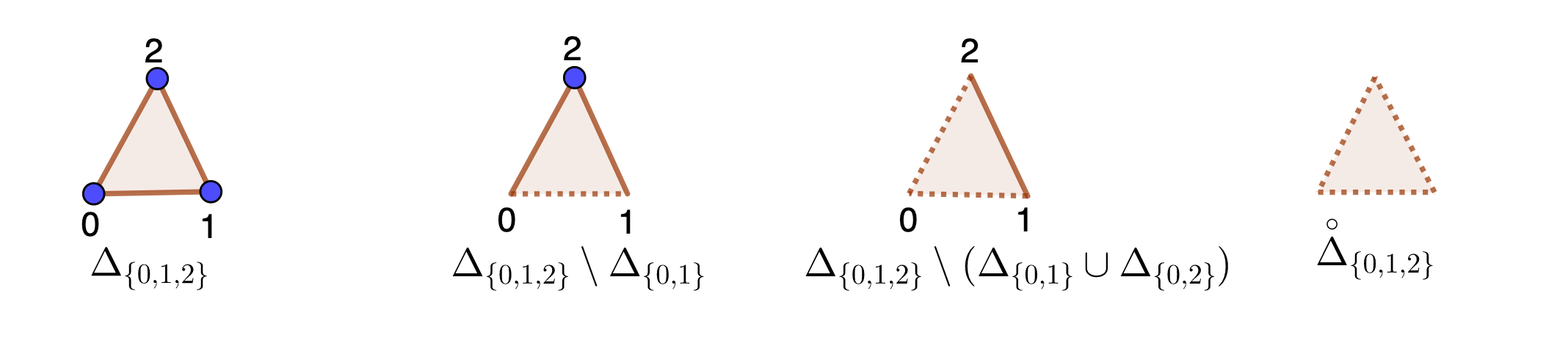}
    \caption{The basic tiles in dimension two.}
    \label{fig2tiles}
      \end{center}
 \end{figure}
\end{ex}

\begin{defi}[Definition $2.4$ of \cite{SaWel2}]
\label{defMorsetile}
A Morse tile of dimension $n$ and order $k \in \{0, \dots , n+1\}$ is an $n$-simplex $\sigma$ deprived of $k$ of its facets together with, if $k \geq 1$, a possibly empty face $\mu$ of higher codimension. The simplex $\sigma$ (resp. $\sigma$ deprived of the $k$ facets) is called the underlying simplex (resp. basic tile), while $\mu$ is called its Morse face. 
\end{defi}

When $k\geq 1$, the dimension of $\mu$ ranges between $k-1$ and $n-2$ and the underlying basic tile has a unique face of dimension $k-1$, see \cite{SaWel2}. Any basic tile is Morse and a Morse tile is said to be critical of index $k$ iff it is of order $k$ and its Morse face has minimal dimension $k-1$, while a closed simplex is critical of index zero. The other Morse tiles are said to be regular. We sometimes denote by $C^n_k$ (resp. $T^{n,l}_k$) a $n$-dimensional critical tile of index $k$ (resp. a $n$-dimensional Morse tile of order $k$ with $l$-dimensional Morse face), so that $C^n_k =T^{n,k-1}_k$. They are all isomorphic to each other via some affine isomorphism. 

\begin{defi}[Definition $2.8$ of \cite{SaWel2}]
\label{defMorsetileable}
A subset $S$ of the underlying topological space $\vert K \vert$ of a finite simplicial complex $K$ is said to be Morse tileable iff it admits a partition by Morse tiles such that for every $j \geq 0$, the union of tiles of dimension $\geq j$ is closed in $S$.
Such a partition $\tau$ is called a Morse tiling and the closure of $S$ in $K$ is called the underlying simplicial complex. 
\end{defi}

When the tiling uses only basic tiles, it is called an $h$-tiling, see \cite{SaWel1, SaWel2}. Of special interest is the case $S = \overline{S}$ where a finite simplicial complex is Morse tiled, but Definition \ref{defMorsetileable} is more general and includes sets such as the triangulated handles $\stackrel{\circ}{\Delta}_k \times \Delta_{n-k}$ of Theorem \ref{theointro4}.
The dimension of a tileable subset is the dimension of the underlying simplicial complex, that is the maximal dimension of the tiles in any Morse tiling. When all tiles have same dimension, the tiling is said to be pure dimensional. This is always the case on compact triangulated manifolds. Indeed,

\begin{lemma}
\label{lemmapure}
Any Morse tiling on a compact connected triangulated manifold is pure dimensional. 
\end{lemma}

\begin{proof}
Let $n$ be the dimension of the triangulated manifold $K$ and let $\tau$ be any Morse tiling on $K$. Then, the $n$-dimensional tiles of $\tau$ cover all open $n$-dimensional simplices of $K$ and by Definition \ref{defMorsetileable}, their union is closed in $K$, so that it contains all closed $n$-dimensional simplices of $K$ as well. Since $K$ is a compact connected triangulated manifold,
the latter is $K$ itself.
\end{proof}

\begin{defi}
\label{defmvector}
The $h$-vector (resp. $c$-vector) of a $n$-dimensional Morse tiling $\tau$ is the vector $h(\tau) = (h_{0} (\tau), \dots , h_{n+1} (\tau))$ (resp. $c(\tau) = (c_{0} (\tau), \dots , c_{n} (\tau))$)
whose entries $h_{k} (\tau)$, $k \in \{0, \dots , n+1\}$ (resp.  $c_{k} (\tau)$, $k \in \{0, \dots , n\}$), are the number of tiles of order $k$ (resp. critical tiles of index $k$) used by $\tau$.
\end{defi}

In particular, $h_0 (\tau) = c_0 (\tau) $ and $ h_{n+1} (\tau) = c_{n} (\tau)$. 
Recall that the $h$-vector $h(\tau_K)$ of any $h$-tiling $\tau_K$ of a finite simplicial complex $K$ coincides with the $h$-vector of $K$ as soon as $h_0 (\tau_K) = 1$ and in any cases, two $h$-tilings $\tau_K$ and $\tau'_K$ of $K$ have same $h$-vector as soon as $h_0(\tau_K) = h_0 (\tau'_K)$, by Theorem $4.9$ and Corollary $4.10$ of \cite{SaWel1}.

\begin{ex}
The boundary of an $n$-simplex admits $h$-tilings using exactly one $(n-1)$-dimensional tile of each order. 
\end{ex}

The $c$-vector, or critical vector, of a Morse tiling encodes the Euler characteristic of the tiled simplicial complex. Indeed,

\begin{lemma}
\label{lemmaEuler}
Let $K$ be an $n$-dimensional finite simplicial complex equipped with a Morse tiling $\tau_K$. Then, its Euler characteristic satisfies $\chi ( K) = \sum_{k=0}^n (-1)^k c_k (\tau_K)$.
\end{lemma}

\begin{proof}
Let us equip the underlying topological space $\vert K \vert$ with its structure of cellular complex given by open simplices and compute $\chi ( K)$ as the alternate sum of the dimensions of its cellular chain complexes.
By Lemma $2.5$ of \cite{SaWel2}, the contribution of each regular Morse tile to this count vanishes while a critical tile of index $k$ contributes as $(-1)^k $. Hence the result. 
\end{proof}

We finally recall the definition of Morse shellability given in \cite{SaWel2}. 

\begin{defi}[Definition $2.14$ of \cite{SaWel2}]
\label{defMorseshellable}
A subset $S$ of  the underlying topological space $\vert K \vert$ of a finite simplicial complex $K$ is said to be Morse shellable iff it admits a Morse tiling together with a filtration $\emptyset \subset S_1 \subset \dots \subset S_N = S$ of Morse tiled subsets such that for every 
$i \in \{2, \dots, N \}$, $S_i \setminus S_{i-1}$ is a single tile of the tiling. 
\end{defi}
A Morse tiled subset of $S$ is a union of tiles which is closed in $S$, see Definition $9$ of \cite{SaWel2}. When the tiling uses only basic tiles, this notion of Morse shelling coincides with the classical notion of shelling, without the non-emptyness assumption though, see \S \ref{subsecsimplcx}, Theorem $2.15$  and Remark $2.16$ of \cite{SaWel2}. A finite simplicial complex, when equipped with a Morse tiling, carries discrete vector fields which are compatible with the tiling and in the case of a Morse shelling, any of these is the gradient vector field of a discrete Morse function in the sense of R. Forman \cite{For1}, whose critical points are in one-to-one correspondence, preserving the index, with the critical tiles of the tiling, see Theorem $1.2$ of \cite{SaWel2}. The Betti numbers of a Morse shelled finite simplicial complex thus get bounded from above by the number of critical tiles of the corresponding index of the shelling, see Corollary $1.5$ of \cite{SaWel2}. 

\subsection{Tame Morse shellings}
\label{subsectame}

In order to get a Morse shelling on the product of two Morse shelled complexes, we need the shellings to satisfy some tameness condition which we now introduce. 

\begin{prop}
\label{proporder}
Let $K$ be a finite simplicial complex whose edges are oriented. Then, the following properties are equivalent:
\begin{enumerate}
\item There is no triangle in $K$ whose boundary is an oriented one-cycle. 
\item For every simplex of $K$, the relation "$x \leq y$ iff $x=y$ or the edge between $x$ and $y$ is oriented from $x$ to $y$" defines a total order on its vertices. 
\end{enumerate}
Moreover, under these conditions, the inclusion of faces define increasing maps, that is they preserve the order on the vertices. 
\end{prop}

\begin{proof}
The second condition implies the first one by transitivity. Indeed, if $x,y,z$ denote the three vertices of a triangle $\theta$ and if the edges are oriented from $x$ to $y$ and from $y$ to $z$, then
by transitivity of the order, $x \leq z$, so that the edge between $x$ and $z$ cannot be oriented from $z$ to $x$, it would imply $x=z$ by antisymmetry and this order wouldn't be total. 
Conversely, let $\sigma$ be any simplex of $K$, the relation defined in the second property is reflexive by definition and antisymmetric since two different vertices $x,y$ are joined by a unique edge so that $x \leq y$
and $y \leq x$ cannot happen unless $x=y$. Now the transitivity follows from the first property. Indeed, if $x,y$ and $z$ are three different vertices of $\sigma$, then we may assume that the edge between $x$ and $y$ is oriented from $x$ to $y$ and that the edge between $y$ and $z$ is oriented from $y$ to $z$. Let $\theta$ be the face with vertices $x,y,z$. By the first property, the edge between $x$ and $z$ has to be oriented from $x$ to $z$. Hence the transitivity. This order relation is then total since any two vertices of $\sigma$ are connected by an edge. 
\end{proof}

\begin{ex}
\label{exorder}
1) If the vertices of a finite simplicial complex are totally ordered, then this order induces an orientation on every edge, from the minimal vertex to the maximal one, and the conditions of Proposition \ref{proporder} get satisfied. 

2) The boundary $\partial \theta$ of a triangle $\theta$ satisfies the properties of Proposition \ref{proporder} whatever the orientations on its edges are, since it contains no two-simplex. However,
$\theta$ itself equipped with such orientations need not satisfy these properties. 
\end{ex}

The first part of Example \ref{exorder} shows that it is always possible to orient the one-skeleton of a finite simplicial complex $K$ in order to define in a compatible way a total order on the vertices of each of its simplices, turning it into an ordered simplicial complex in the sense of Definition $II. 8.7$ of \cite{EilSteen}.  When $K$ is equipped with a Morse tiling, we would like in addition that for every Morse tile with underlying simplex $\sigma$ and non-empty Morse face $\mu$, the vertices of $\mu$ are the maximal ones among the ones of $\sigma$. Recall that the link $\lk_\sigma (\mu)$ of $\mu$ in $\sigma$ is by definition the convex hull of the vertices of $\sigma \setminus \mu$. We thus would like that the edges between $\lk_\sigma (\mu)$ and $\mu$ are oriented from $\lk_\sigma (\mu)$ to $\mu$, see Figure \ref{figTetra}.

\begin{figure}[h]
   \begin{center}
    \includegraphics[scale=0.3]{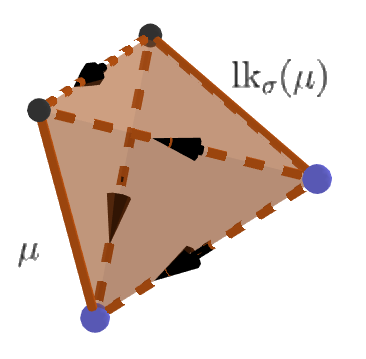}
    \caption{A Morse face $\mu$ in a three-simplex.}
    \label{figTetra}
      \end{center}
 \end{figure}

\begin{defi}
\label{deftame}
The tiling of a Morse tiled set $S$ is said to be tame iff there exists an orientation on the one-skeleton of the underlying simplicial complex $K$ which satisfies the following order and tameness conditions.
\begin{enumerate}
\item There is no triangle in $K$ whose boundary is an oriented one-cycle. 
\item For every Morse tile with underlying simplex $\sigma$ and non-empty Morse face $\mu$, the edges between $\lk_\sigma (\mu)$ and $\mu$ are oriented from $\lk_\sigma (\mu)$ to $\mu$.
\end{enumerate}
\end{defi}
Every $h$-tiling is tame by the first part of Example  \ref{exorder}, since the second condition of Definition \ref{deftame} is then empty and the first one satisfied. In fact, the order condition in Definition \ref{deftame} provides a structure of ordered simplicial complex on $K$ given by Proposition \ref{proporder}, see Definition $II. 8.7$ of \cite{EilSteen}, while the tameness condition requires some compatibility between this structure and the tiling. 

The tameness property gets satisfied by any Morse tiling after one barycentric subdivision for example.

\begin{prop}
\label{propbarycsubdtame}
The first barycentric subdivision of any Morse tiled (resp. shelled) set carries tame Morse tilings (resp. tame Morse shellings) containing the same number of critical tiles with the same indices. 
\end{prop}

\begin{proof}
Let $S$ be a Morse tiled (resp. Morse shelled) set and let $K$ be its underlying simplicial complex. By Corollary $2.21$ of \cite{SaWel2}, the first barycentric subdivision $\Sd (S)$ of $S$ carries Morse tilings (resp. shellings) having the same number of critical tiles with the same indices. Its underlying simplicial complex is $\Sd (K)$. Now, the one-skeleton of the latter is canonically oriented. Indeed, its vertices are by definition the barycenters $\hat{\sigma}$ of the simplices $\sigma$ of $K$ while an edge connects two vertices $\hat{\sigma}$ and $\hat{\tau}$ iff $\sigma$ is a face of $\tau$ or vice-versa, see  \cite{Munk}. Let us orient such an edge from $\hat{\sigma}$ to $\hat{\tau}$ iff $\tau$ is a face of $\sigma$. The order condition of Definition \ref{deftame} gets satisfied by this order on the one-skeleton of $\Sd (K)$.
We have to prove that the tameness condition gets satisfied as well. Let $T'$ be a Morse tile of $\Sd (S)$ with underlying simplex $\sigma'$ and non-empty Morse face $\mu'$. By construction, there exists a Morse tile $T$ of $S$, with underlying simplex $\sigma$ and non-empty Morse face $\mu$ such that $\sigma' \subset \Sd (\sigma)$ and $\mu' = \Sd (\mu) \cap \sigma'$, see \cite{SaWel2}.
There exists then a maximal flag $\sigma_0 \subset \sigma_1 \subset \dots \subset \sigma_n = \sigma$ such that the vertices of $\sigma'$ are the barycenters $\hat{\sigma}_i$ of $\sigma_i$, $i \in \{ 0, \dots , n \}$, where $n$ denotes the dimension of $T'$. By maximal flag we mean that for every $0 \leq i \leq n$, $\dim \sigma_i = i$ and  for every $0 \leq i < j \leq n$, $\sigma_i $ is a face of $\sigma_j$. Now such a vertex $\hat{\sigma}_j $ belongs to $\mu'$ iff $\sigma_j $ is a face of $\mu$ and this then implies that $\hat{\sigma}_i \in \mu'$ for all $i \geq j$, since then $\sigma_i \subset \sigma_j$.
The vertices of $\sigma'$ that belong to the Morse face $\mu'$ are thus the maximal ones with respect to this canonical order. Hence the result. 
\end{proof}

We may finally provide a criterium which ensures that a Morse tiling is tame. 

\begin{prop}
\label{proptame}
A Morse tiling is tame if $\sigma \cap \mu' = \sigma' \cap \mu$ for every tiles $T, T'$ with underlying simplices $\sigma, \sigma'$ and non-empty Morse faces $\mu, \mu'$. 
\end{prop}

For example, the condition in Proposition \ref{proptame} gets satisfied if all tiles with non-empty Morse faces have disjoint underlying simplices. 

\begin{proof}
Let $K$ be the underlying simplicial complex of such a Morse tiled set $S$ and let $L$ be the union of all Morse faces of its tiles. Let us fix a total order on the vertices of $L$ and a total order on the vertices of $K \setminus L$. 
Then, all edges of $K$ whose vertices are both in $L$ or both outside $L$ get oriented by these total orders from the minimal vertex to the maximal one. We finally orient the edges between $K \setminus L$ and $L$ from $K \setminus L$ to $L$. These orientations satisfy the properties of Proposition \ref{proporder}, compare Example \ref{exorder}. Moreover, every tile $T$ with underlying simplex $\sigma$ and non-empty Morse face $\mu$ satisfies $\sigma \cap L = \mu$ by hypothesis
and its edges between $\lk_\sigma (\mu)$ and $\mu$ are oriented from $\lk_\sigma (\mu)$ to $\mu$ by construction, so that the tiling is indeed tame by definition. 
\end{proof}

\section{Main results}
\label{secmain}

\subsection{Palindromic vectors}
\label{subsecpalindrome}

Let $n$ be a non-negative integer, the involution $k \in \{ 0 , \dots , n \} \mapsto n-k \in \{ 0 , \dots , n \} $ induces the automorphism $v = (v_0 , \dots , v_n) \in \R^n \mapsto \check{v} = (v_n , \dots , v_0) \in \R^n$.

\begin{defi}
\label{defpalind}
A vector $v$ of $\R^n$ is said to be palindromic iff $\check{v} = v$.
\end{defi}

For example, the real Betti numbers of closed connected oriented manifolds define palindromic vectors by Poincaré duality, see \cite{Munk, BottTu}. The $h$-vectors of convex polytopes are palindromic as well, see \cite{McMu, Stan, Fult, Z}. We are going to prove that $h$-vectors of Morse tilings are likewise often palindromic. 

\begin{theo}
\label{theohvector}
Let $K$ be an $n$-dimensional simplicial complex homeomorphic to a closed manifold and equipped with an $h$-tiling $\tau$. Then,
\begin{enumerate}
\item If $n$ is odd, the $h$-vector of $\tau$ is palindromic.
\item If $n$ is even, for every $i \in \{ 0 , \dots , n+1 \}$, 
$$h_i (\tau) - h_{n+1-i} (\tau)  = (-1)^i {n+1 \choose i} (h_0 (\tau) - h_{n+1} (\tau) ).$$
\end{enumerate}
\end{theo}

\begin{proof}
By Theorem $4.9$ of \cite{SaWel1}, the $h$-vector of $\tau$ satisfies $\sum_{i=0}^{n+1} h_i (\tau) X^{n+1-i} = \sum_{i=0}^{n+1} f_{i-1} (K) (X-1)^{n+1-i} $ provided one sets $ f_{-1} (K) = h_0 (\tau)$. 
By Theorem $2.1$ of \cite{Mac}, the Dehn-Sommerville relations can be expressed by the relation $R_K (-1-X) = (-1)^{n+1} R_K (X)$, where $R_K (X) = X \sum_{i=0}^n f_i (X) X^i - \chi(K) X$, see also Theorem $1.1$
of \cite{SaWelDCG}. Finally, we know from Lemma \ref{lemmaEuler} that the Euler characteristic of $K$ satisfies $\chi (K) = h_0 (\tau) + (-1)^n h_{n+1} (\tau)$, since the only critical tiles of an $h$-tiling are the open and closed simplices. We deduce that  $R_K (X) = \sum_{i=0}^{n+1} h_i (\tau) X^i (X+1)^{n+1-i} - h_0 (\tau) - \chi (K) X$, so that Macdonald's result \cite{Mac} becomes
\begin{eqnarray*}
\sum_{i=0}^{n+1} \big( h_i (\tau)  - h_{n+1-i} (\tau) \big) X^i (X+1)^{n+1-i} - \big(  h_0 (\tau) - h_{n+1} (\tau) \big) = 0 \text{ if } n \text{ is even and}\\
\sum_{i=0}^{n+1} \big( h_i (\tau)  - h_{n+1-i} (\tau) \big) X^i (X+1)^{n+1-i} - \big(  h_0 (\tau) - h_{n+1} (\tau) \big)(1+2X) = 0 \text{ if } n \text{ is odd.}
\end{eqnarray*}
If $n$ is odd, the Euler characteristic of a closed $n$-dimensional manifold vanishes from Poincaré duality, see \cite{Munk} for example, so that $h_0 (\tau) = h_{n+1} (\tau) $. We thus deduce in this case that
$h (\tau) $ is palindromic.  If $n$ is even, we observe that $1 = \sum_{i=0}^{n+1} (-1)^i {n+1 \choose i} X^i (X+1)^{n+1-i} $ and get the result, since the monomials $(X^i (X+1)^{n+1-i} )_{i \in \{ 0 , \dots , n+1\}}$ are linearly independant over $\R$. 
\end{proof}

\begin{cor}
\label{corpalindromic}
The $h$-vector of any $h$-tiling on a simplicial complex homeomorphic to a closed manifold is palindromic iff its $c$-vector is palindromic. 
\end{cor}

\begin{proof}
The $h$-tiling $\tau$ of  a simplicial complex homeomorphic to a closed $n$-dimensional manifold only contains $n$-dimensional tiles by Lemma \ref{lemmapure} and the singular ones are the open and closed simplices by definition. 
Thus, if $h (\tau)$ is palindromic, $h_0 (\tau) = h_{n+1} (\tau)$, so that $c_0 (\tau) = c_{n} (\tau)$ which means that $c (\tau)$ is palindromic as well. Converserly, if $c (\tau)$ is palindromic, then $c_0 (\tau) = c_{n} (\tau)$, so that $h_0 (\tau) = h_{n+1} (\tau)$ and the result follows from Theorem \ref{theohvector}.
\end{proof}

\begin{ex}
\label{exhtiling}
1) The boundary of a $(n+1)$-simplex is shellable and the associated $h$-tiling uses one $n$-dimensional basic tile of each order. Its $h$-vector $(1, \dots , 1)$ is thus palindromic. 

2) The boundary of a triangle is also tiled by three one-dimensional tile of order one, see Figure \ref{figExotic}. The associated $h$-vector $(0,3,0)$ is palindromic. 

\begin{figure}[h]
   \begin{center}
    \includegraphics[scale=1]{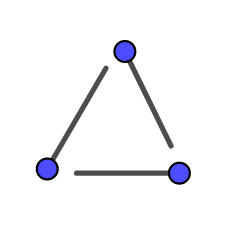}
    \caption{The non-shellable tiling on $\partial \Delta_2$.}
    \label{figExotic}
      \end{center}
 \end{figure}

3) The cylinder $\Delta_1 \times \partial \Delta_2$ has a triangulation tiled by six basic tiles of order one, obtained by gluing three copies of the square $\Delta_1 \times T^1_1$ pictured in Figure \ref{figCylinder}. 
By caping this cylinder with two open triangles, we get an $h$-tiled triangulation on the two-sphere for which neither the $h$-vector $(0,6,0,2)$ nor the critical vector $(0,0,2)$ are palindromic. 

\begin{figure}[h]
   \begin{center}
    \includegraphics[scale=1]{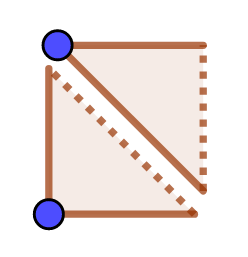}
    \caption{An $h$-tiling on $\Delta_1 \times T^1_1$.}
    \label{figCylinder}
      \end{center}
 \end{figure}
\end{ex}

Example \ref{exhtiling} exhibits in particular an $h$-tiling with non-palindromic $h$-vector on the triangulated two-sphere.  Does there exist such $h$-tilings on the other even-dimensional spheres?

In the case of Morse tilings, we observe. 

\begin{theo}
\label{theomvector}
Let $K$ be an $n$-dimensional simplicial complex homeomorphic to a closed manifold of dimension at most three and equipped with a Morse tiling $\tau$. Then, the following three conditions are equivalent.
\begin{enumerate}
\item The tiling $\tau$ uses as many open simplices as closed simplices. 
\item The $c$-vector of $\tau$ is palindromic.
\item The $h$-vector of $\tau$ is palindromic.
\end{enumerate}
\end{theo}

\begin{proof}
The implications $2 \Rightarrow 1$ and $3 \Rightarrow 1$ hold true in any dimension, while $1 \Rightarrow 2$ is obvious in dimension at most two and follows from Lemma \ref{lemmaEuler} in dimension three, since the Euler characteristic of $K$ then vanishes from Poincaré duality, see \cite{Munk}. Let us thus assume that $1$ holds true and prove the implication $1 \Rightarrow 3$. By the simplest Dehn-Sommerville relation, every $(n-1)$-simplex $\sigma$ of $K$ is the face of exactly two $n$-simplices. The interiors of these two simplices are covered by two tiles of the tiling and the open face 
$\stackrel{\circ}{\sigma}$ is a facet of one of them and a missing facet of the other since the tiling defines a partition of $K$. It follows that the total number of facets of the tiles of $\tau$ coincides with the total number of missing facets of these tiles.  By $1$, the contributions to these totals of the tiles of order $0$ and $n+1$ coincide. If $n=3$, the same holds true for tiles of order two since they have both two facets and two missing facets. In dimension two (resp. three), we deduce that $\tau$ uses as many tiles of order one as tiles of order two (resp. three), so that $h(\tau)$ is palindromic. 
\end{proof}

\begin{ex}
\label{exoctahedron}
The octahedron carries a Morse tiling with non palindromic  $c$-vector and $h$-vector. It is obtained by patching the two tiled squares pictured in Figure \ref{figOctahedron}, so that its $h$-vector (resp. $c$-vector) equals
$(1,4,1,2)$ (resp. $(1,1,2)$).

\begin{figure}[h]
   \begin{center}
    \includegraphics[scale=1]{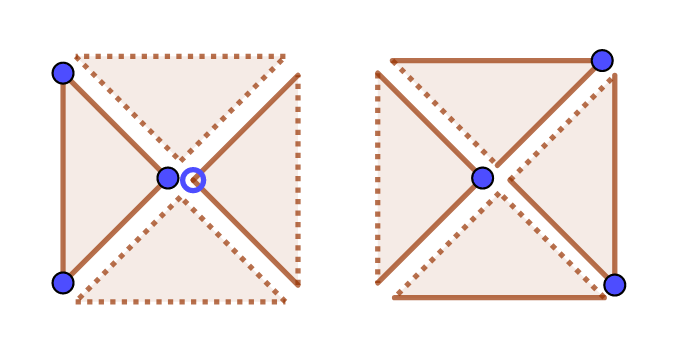}
    \caption{A Morse tiling on the octahedron.}
    \label{figOctahedron}
      \end{center}
 \end{figure}

\end{ex}
The tiling given by Example \ref{exoctahedron} contains two critical tiles of indices one and two which could be replaced by two regular tiles of order two to produce a tiling with palindromic $h$-vector. Such examples
with non-palindromic $h$-vectors and $c$-vectors can be obtained in a similar way in higher dimensions. But we do not know any Morse tiled closed triangulated manifold with palindromic critical vector and non-palindromic $h$-vector.

\subsection{Tilings of products}
\label{subsecproducts}

Recall that the product of two simplicial complexes is not a simplicial complex, it is a polyhedral complex whose cells are products of two simplices. Such a product can nevertheless be triangulated in such a way that each product of two simplices becomes the union of simplices of the underlying affine space, see \S $II. 8$ of \cite{EilSteen}, \cite{GKZ, San} and \S \ref{subsectriangulations}. We are going to consider such triangulations, which are primitive in the sense of Definition \ref{defprimitive} and in fact associated to staircases, see \S \ref{subsecstraircases} and \cite{GKZ}.

\begin{defi}
\label{defprimitive}
A primitive triangulation of a polyhedral complex is a triangulation having the same set of vertices.
\end{defi}

Our main result is the following Theorem \ref{theoproduct}, where we denote by $uv$ the graded product of a vector $u=(u_0 , \dots , u_n)$ of $\R^{n+1}$ with a vector $v=(v_0 , \dots , v_m)$ of $\R^{m+1}$, that is the product of the corresponding polynomials, so that $uv = (w_0, \dots , w_{n+m})$ where for every
$k \in \{ 0 , \dots , n+m \}$, $w_ k = \sum_{j=0}^k u_jv_{k-j}$. 

\begin{theo}
\label{theoproduct}
Let $S_1$ and $S_2$ be two Morse tiled (resp. shelled) sets with tame tilings (resp. shellings) $\tau_1$ and $\tau_2$. Then, $S_1 \times S_2$ carries tame Morse tileable (resp. shellable) primitive triangulations with critical vector $c(\tau_1)c(\tau_2)$. Moreover, if $\tau_1$, $\tau_2$ are pure dimensional, their $h$-vector is palindromic provided $h(\tau_1)$ and $h(\tau_2)$ are. 
\end{theo}

By triangulation of $S_1 \times S_2$ we mean triangulations on the product of the underlying simplicial complexes. We do not guarantee Morse tileability for all primitive triangulations on this product, the ones for which we do by Theorem \ref{theoproduct} are given by the tameness of the tilings, see \S \ref{subsecprooftheoproduct}.  Let us also recall that all $h$-tilings are pure dimensional  in the case of triangulated manifolds by Lemma \ref{lemmapure}. We deduce Theorem \ref{theointro2}, namely.

\begin{cor}
\label{corproductmanifolds}
Every finite product of closed manifolds of dimensions less than four carries triangulations which admit tame Morse shellings with palindromic $c$-vectors and $h$-vectors. 
\end{cor}

\begin{proof}
By Theorem $1.4$ of \cite{SaWel2}, every closed connected manifold of dimension less than four carries a Morse shellable triangulation which can moreover be chosen in such a way that the Morse shelling uses a unique critical tile of index $0$ and a unique critical tile of maximal index, for there exists a Morse function on this manifold having a single minimum and a single maximum, see \cite{Cerf}. By Proposition \ref{propbarycsubdtame}, such a Morse shelling becomes tame after one barycentric subdivision and it keeps the property to use only one closed and one open simplex. By Theorem \ref{theomvector}, its critical and h-vector are then palindromic. The result now follows by finite induction from Theorem \ref{theoproduct}.
\end{proof}

When one of the tilings $\tau_1$, $\tau_2$ uses only regular tiles, the tilings given by Theorem \ref{theoproduct} share the same property since $c(\tau_1)c(\tau_2)$ then vanishes. However, it is not supposed to be an $h$-tiling even if $\tau_1$ and $\tau_2$ are, for $c(\tau_1)c(\tau_2)$ has more than two non-vanishing entries in general. The following variant of Theorem \ref{theoproduct} fills this gap. 

\begin{theo}
\label{theoproducthtiling}
Let $S_1$, $S_2$ be two $h$-tiled sets and let the one-skeleton of their underlying simplicial complexes $K_1$, $K_2$ be equipped with orientations given by Proposition \ref{proporder}. 
Then, the tiling of $S_1 \times S_2$ given by Theorem \ref{theoproduct} is an $h$-tiling provided that if $S_1$ (resp. $S_2$) contains a tile which has been deprived both of its facet not containing the biggest vertex and its facet not containing the least vertex, then every tile of $S_2$ (resp. $S_1$) has been deprived either of its facet not containing the biggest vertex or of its facet not containing the least vertex. 
\end{theo}
Recall that the total orders on the vertices of the simplices of $K_1$ and $K_2$ are given by Proposition \ref{proporder}. If the $h$-tiling of $S_1$ (resp. $S_2$) contains an open simplex, then the $h$-tiling of $S_2$ (resp. of $S_1$) has in particular to be regular for Theorem \ref{theoproducthtiling} to apply. Now, if $S_2$ is the non-shellable tiling of $\partial \Delta_2$ given in the second part of Example \ref{exhtiling}, then Theorem \ref{theoproducthtiling} applies whatever $S_1$ is. We may compute the $h$-vector of $S_1 \times S_2$ in this case.

\begin{theo}
\label{theodelta2}
Let $S$ be an $h$-tiled set of pure dimension $n$. Then, $S \times \partial \Delta_2$ carries primitive triangulations which admit $h$-tilings $\tau$ such that $h_0 (\tau) = h_{n+2} (\tau) = 0$ and 
for every $j \in \{ 1 , \dots , n+1 \}$, $h_j (\tau) = j h_j (S) + (n+2-j) h_{j-1} (S)$.
\end{theo}
In Theorem \ref{theodelta2} again, $h(\tau)$ is then palindromic as soon as $h(S)$ is, or provided $S$ is homeomorphic to a closed manifold by Corollary \ref{corpalindromic}, since $c(\tau)$ is palindromic. This result implies Theorem \ref{theointro1}, showing that the product of a sphere and a torus of any dimensions carries $h$-tileable triangulations. More precisely, we deduce the following Corollary \ref{cortores}.

\begin{defi}
\label{defwalks}
A walk of length $m$ from the integer $a$ to $b$ is a sequence $w_{n,m} = \big( w_{n,m} (i) \big)_{i \in \{n, \dots , n+m\}}$ such that $w_{n,m} (n) = a$, $w_{n,m} (n+m) = b$ and for every $i \in \{n, \dots , n+m-1\}$,
either $w_{n,m} (i+1) = w_{n,m} (i) $ or $w_{n,m} (i+1) = w_{n,m} (i) +1$. The weight of such a walk is the product $p(w_{n,m}) = \Pi_{i=n}^{n+m-1} p_i(w_{n,m})$, where $p_i(w_{n,m}) = w_{n,m} (i) $ if $w_{n,m} (i+1) = w_{n,m} (i) $ and $p_i(w_{n,m}) = i+1-w_{n,m} (i) $ if $w_{n,m} (i+1) = w_{n,m} (i) +1$.
\end{defi}

\begin{cor}
\label{cortores}
For every $n \geq 0$ and $m \geq 1$, the product $\partial \Delta_{n+1} \times (\partial \Delta_2 )^m$ carries primitive triangulations which admit $h$-tilings using no critical tile. Moreover, the $h$-vector of such an $h$-tiling $\tau$ is palindromic and satisfies, for every $k \in \{ 0 , \dots , n+m+1 \}$, $h_k (\tau) = \sum_{w_{n,m}} p(w_{n,m})$, where the sum is taken over all walks of length $m$ from an element of $ \{ 0 , \dots , n+1 \}$ to $k$.
\end{cor}

\begin{proof}
Let us equip $\partial \Delta_{n+1}$ with the shelling given in the first part of Example \ref{exhtiling}, which uses one tile of each order $k \in  \{ 0 , \dots , n+1 \}$. The result follows by induction on $m$, by successive applications of Theorem \ref{theodelta2}. Indeed, when $m=1$, if $k \in \{ 0, n+2 \}$, the walk is unique and its weight vanishes. If $k$ belongs to $\{ 1, \dots , n+1 \}$, there are two walks leading to $k$, namely $(k,k)$ and $(k-1,k)$. The weight of the first one is $k$ by Definition \ref{defwalks} and the weight of the second one is $n+2-k$. By Theorem \ref{theodelta2}, $h_k (\partial \Delta_{n+1} \times \partial \Delta_2) = k + (n+2-k)$ coincides with the sum of these walks. Let us now assume that the result is proven up to the rank $m$ and let us prove it for $m+1$. Again, for $k=0$ or $n+m+2$, the walk leading to $k$ is unique and its weight vanishes. When $k$ belongs to $\{ 1, \dots , n+m+1 \}$, a walk $w_{n,m+1}$ leading to $k$ restricts either to a walk $w_{n,m}$ leading to $k$ or to a walk $w_{n,m}$ leading to $k-1$ while conversely, any such walk extends uniquely to a walk $w_{n,m+1}$ leading to $k$. Moreover, by Definition \ref{defwalks}, in the first case, $p(w_{n,m+1}) = p(w_{n,m}) k$ and in the second, $p(w_{n,m+1}) = p(w_{n,m}) (n+m+2-k)$, while by the induction hypothesis, $\sum_{w_{n,m}} p(w_{n,m}) = h_k(\partial \Delta_{n+1} \times (\partial \Delta_2 )^m)$ (resp. $\sum_{w_{n,m}} p(w_{n,m}) = h_{k-1} (\partial \Delta_{n+1} \times (\partial \Delta_2 )^m)$) in the first case (resp. in the second case). The result now follows from Theorem  \ref{theodelta2} with $S = \partial \Delta_{n+1} \times (\partial \Delta_2 )^m$, the palindromic property being ensured by Corollary \ref{corpalindromic}.
\end{proof}

The $h$-vector $h^{n,m}$ of the $h$-tilings given by Corollary \ref{cortores} has its own interest. By Theorem $4.9$ of \cite{SaWel1}, it does not depend on the choice of the tiling but differs from the $h$-vector of the underlying primitive triangulation, see Corollary $4.10$ of \cite{SaWel1}. What is the asymptotic of $h^{n,m}$ as $m$ grows to $+ \infty$? More precisely, what is the limit, as $m$ grows  to $+ \infty$, of the probability measure
$$\frac{1}{\sum_{k=0}^{n+m+1} h^{n,m}_k} \sum_{k=0}^{n+m+1} h^{n,m}_k \delta_{\frac{2k -n-m-1}{n+m+1}},$$
where $\delta_x$ denotes the Dirac measure at $x \in \R$?

We also do not know which closed manifolds carry $h$-tileable triangulations. They have non-negative Euler characteristic by Lemma \ref{lemmaEuler}.

\subsection{Tilings of handles and duality}
\label{subsechandles}

The sets $S_1 , S_2$ in Theorems \ref{theoproduct} and \ref{theoproducthtiling} may just consist of single tiles. In fact, these results follow from this special case to which we devote this section.

\begin{theo}
\label{theoproducttiles}
Let $T_1$ and $T_2$ be two basic tiles, one of which being regular. Then, $T_1 \times T_2$ carries shellable primitive triangulations using only regular tiles in their shelling. If $T_1$ and $T_2$ are Morse tiles,
then $T_1 \times T_2$ carries Morse shellable primitive triangulations using a critical tile iff both $T_1$ and $T_2$ are critical and this critical tile is then unique of index the sum of the indices of $T_1$ and $T_2$. 
Moreover, all these shellings are tame and pure dimensional.
\end{theo}
These triangulations and shellings given by Theorem \ref{theoproducttiles} are inherited from particular total orders on the vertices of the underlying simplices of $T_1$ and $T_2$, see \S \ref{subseccond}. They all have same $h$-vector and satisfy the following duality property.

\begin{defi}
\label{defdualtiles}
Let $\sigma$ be an $n$-simplex and $\sigma_0, \dots , \sigma_n$ be its facets. 
For every $J \subset \{ 0 , \dots , n \}$, the basic tiles $T = \sigma \setminus \bigcup_{j \in J} \sigma_j$ and $\widecheck{T} = \sigma \setminus \bigcup_{j \in  \{ 0 , \dots , n \} \setminus J} \sigma_j$ are said to be dual to each other. 
\end{defi}

\begin{theo}
\label{theoduality}
Let $T_1 , \widecheck{T}_1$ (resp. $T_2 , \widecheck{T}_2$) be two Morse tiles whose underlying basic tiles are dual to each other. Then, the Morse shellings on $T_1 \times T_2$ (resp. $\widecheck{T}_1 \times \widecheck{T}_2$) given by Theorem \ref{theoproducttiles} all have same $h$-vector and satisfy $h(\widecheck{T}_1 \times \widecheck{T}_2) = \widecheck{h}(T_1 \times T_2)$.
\end{theo}
In Theorem \ref{theoduality}, $\widecheck{h}(T_1 \times T_2)$ denotes the image of the $h$-vector $h(T_1 \times T_2)$ under the palindromic automorphism defined in \S \ref{subsecpalindrome}.

Theorem \ref{theoproducttiles} provides in particular Morse shellings on every handle, whatever its index is, where by handle of index $k$ and dimension $n$, we mean the product $\stackrel{\circ}{\Delta}_k \times \Delta_{n-k}$ as defined in \cite{SaWel2}. Such a Morse shelling has already been obtained in index $1$ and $n-1$, see Corollary $3.17$ of \cite{SaWel2}.

\begin{cor}
\label{corhandle}
For every $0 \leq k \leq n$, the handle $\stackrel{\circ}{\Delta}_k \times \Delta_{n-k}$ carries Morse shellable primitive triangulations using a unique critical tile, of index $k$. $\square$
\end{cor}

Recall that Theorem $1.4$ of \cite{SaWel2}, which provides Morse shelled triangulations on every closed three-manifold, has been obtained by successive attachments of such Morse shelled triangulated handles. We end this section by giving other remarkable shellings  given by Theorem \ref{theoproducttiles}.

\begin{cor}
\label{corisomtiles}
For every $m,n >0$, $T_n^n \times \Delta_m$ (resp. $T_1^n \times \stackrel{\circ}{\Delta}_m$) carries shellable primitive triangulations using basic tiles which are all isomorphic to each other, of order $n$ (resp. of order $m+1$). 
\end{cor}
Recall that $T_n^n$ (resp. $T_1^n$) denotes a basic tile of dimension $n$ and order $n$ (resp. order one), see \S \ref{subsecMorseshel}.

\begin{cor}
\label{corpalindromictiling}
Let $T_1$ (resp $T_2$) be a Morse tile of odd dimension $n$ (resp. $m$) and of order $\frac{m+1}{2}$ (resp. $\frac{n+1}{2}$). Then, $T_1 \times T_2$ carries Morse shellable primitive triangulations with palindromic $h$-vector. Moreover, if the tiles are basic, these Morse shellings can be chosen to be shellings. 
\end{cor}

\begin{proof}
Theorem \ref{theoproducttiles} provides a Morse shelled triangulation on $T_1 \times T_2$ and even a shelled triangulation if these tiles are basic, since they are regular.
By hypothesis, the basic tile underlying $T_1$ is isomorphic to its dual and likewise, the basic tile underlying $T_2$ is isomorphic to its dual. The $h$-vectors of $T_1 \times T_2$ and $\widecheck{T}_1 \times \widecheck{T}_2$ thus coincide, while they are dual to each other by Theorem \ref{theoduality}. They must then be palindromic. 
\end{proof}

\section{Cartesian product of two simplices}
\label{sectriangulations}

The cartesian product of two simplices is a structure of simplicial complex on their product which is inherited from total orders on their vertices, see Definition $II. 8.8$ of \cite{EilSteen}. We study these primitive triangulations in this section, whose simplices are associated to staircases, together with their shellings, see \S \ref{subsectriangulations}. The Cayley trick helps to visualize them, via the mixed decompositions induced on one of the simplices, see \S \ref{subsecmixeddecomp} and \cite{San, HRS}.

\subsection{Staircases}
\label{subsecstraircases}

Let $m,n$ be two non-negative integers. We denote by $C(n,m)$ the set of increasing -not strictly increasing- functions $f : \{ 0, \dots , n \} \to \{ 0, \dots , m \}$ such that $f(n)=m$ and by $N(n,m)$ its cardinality. We recall that.

\begin{lemma}
\label{lemmadim}
For every $m,n \geq 0$, $N(n,m) = {n+m \choose n}.$
\end{lemma}

\begin{proof}
To every $f \in C(n,m)$ we may associate $\tilde{f} : k \in \{ 0, \dots , n \}  \mapsto f(k)+k \in \{ 0, \dots , m+n \}$. This correspondence between $C(n,m)$ and the set of strictly increasing functions $\{ 0, \dots , n \} \to \{ 0, \dots , m+n \}$ such that $f(n)=m+n$ is bijective. Moreover, the image of such a function is a subset of $\{ 0, \dots , m+n \}$ containing $m+n$ and of cardinality $n+1$ while every subset sharing these properties defines a strictly increasing map $\{ 0, \dots , n \} \to \{ 0, \dots , m+n \}$ such that $f(n)=m+n$. The result follows. 
\end{proof}

The space $C(n,m)$ is equipped with the involution $f \in C(n,m) \mapsto \check{f} \in C(n,m)$, where for every $j \in \{ 0 , \dots , n-1 \}$, $\check{f} (j) = m - f(n-1-j)$ and $\check{f} (n) = m$. Also, the lexicographic order on the n-tuples of integers induces a total order on $C(n,m)$, so that for every $f,g \in C(n,m)$, $f \leq g$ iff $(f(0) , \dots , f(n-1)) \leq (g(0), \dots , g(n-1))$. The minimum of $C(n,m)$ is thus a function which vanishes on 
$\{ 0 , \dots , n-1 \}$ while its maximum is the constant function equal to $m$.

Likewise, we denote by $I(n,m)$  the set of collections $I=(I_j)_{j \in \{ 0, \dots , n \} }$ of intervals $I_j = \{ i \in \{ 0, \dots , m \} \, \vert \, b_I (j) \leq i \leq  e_I (j) \}$ which cover $\{ 0, \dots , m \}$ and satisfy $e_I (j) = b_I (j+1)$ for every $0 \leq j < n$. In particular, $b_I (0) = 0$ and $e_I (n) = m$. This space of staircases, see \S $7.3.D$ of \cite{GKZ}, is equipped with the involution $I=(I_j)_{j \in \{ 0, \dots , n \}} \mapsto \check{I} =(\check{I}_j)_{j \in \{ 0, \dots , n \}}$, where for every $j \in \{ 0, \dots , n \}$, $\check{I}_j = \{ m - e_I (n-j) , \dots , m - b_I (n-j) \}$, so that $b_{\check{I}} (j) = m - e_I (n-j) $ and $e_{\check{I}} (j) = m - b_I (n-j) $. This space also inherits a total order from the lexicographic order, so that for every $I,J \in I(n,m)$, $I \leq J$ iff $(e_I (0) , \dots , e_I (n-1)) \leq (e_J (0) , \dots , e_J (n-1)) $.

These spaces of functions and staircases are in bijective correspondence. Namely, for every $f \in C(n,m)$, let us denote by  $I^f$  the element of $I(n,m)$ such that $e_{I^f}  = f$.

\begin{lemma}
\label{lemmafI}
The maps $f \in C(n,m) \mapsto I^f \in I(n,m)$ and $I \in I(n,m) \mapsto e_I \in C(n,m) $ are bijective, $\Z/2\Z$-equivariant, order preserving and inverse one with respect to the other.
\end{lemma}

\begin{proof}
The maps are order preserving and inverse one with respect to the other by definition. They are thus bijective as well. Now, let $f \in C(n,m)$, we have to check that $I^{\check{f}} = \widecheck{I^f}$. For every $j \in \{ 0, \dots , n \}$, $e_{I^{\check{f}}} (j) = \check{f} (j) = m-f(n-1-j)$ while $e_{\widecheck{I^f}} (j) = m - b_{I^f} (n-j)  = m - e_{I^f} (n-1-j) = m-f(n-1-j)$. Hence the result. 
\end{proof}

Let us finally observe that exchanging the roles of $n$ and $m$ defines an involution $I(n,m) \to I(m,n)$. 

\begin{lemma}
\label{lemmaJi}
For every $I \in I(n,m)$ and every $i \in \{0, \dots , m\}$, set $J_i = \{ j \in \{0, \dots , n\} \, \vert \, i \in I_j \}$. Then, $J=(J_i)_{i \in \{0, \dots , m\}}$ belongs to $I(m,n)$ and the correspondence $I \in I(n,m) \mapsto J \in I(m,n)$ is bijective.
\end{lemma}

\begin{proof}
Let $I \in I(n,m)$ and $i \in \{0, \dots , m\}$. We denote by $b_J (i)$ (resp. $e_J (i)$) the least (resp. greatest) element of $J_i$. If $j \in \{ b_J (i) , \dots , e_J (i) \}$, then $b_I (j) \leq b_I (e_J (i)) \leq i$ and $e_I (j) \geq e_I (b_J (i)) \geq i$, so that $b_I (j) \leq i \leq e_I (j)$, that is $i \in I_j$. We deduce that $J_i$ is the interval $\{ b_J (i) , \dots , e_J (i) \}$. Moreover, $b_J (0) = 0$ and $e_J (m) = n$ by definition. Finally, if $j \in J_i$, then either $i = e_I (j)$ and $j+1 = b_I (j+1) \in J_i$ provided $j < n$, or $i < e_I (j)$ and $j = e_J (i)$. It follows that if $i < m$, $b_J (i+1) = e_J (i)$ so that $J \in J(m,n)$. Now, this correspondence $I \in I(n,m) \mapsto J \in I(m,n)$ is bijective. The preimage of an element $J=(J_i)_{i \in \{0, \dots , m\}}$ is the staircase $(I_j)_{j \in \{0, \dots , n\}}$ defined in a similar way, namely for every $j \in \{0, \dots , n\}$, $I_j = \{ i \in \{0, \dots , m\} \, \vert \, j \in J_i \}$. Indeed, we check likewise that $I \in I(n,m)$ and for every $(j,i) \in \{0, \dots , n\} \times \{0, \dots , m\}$, the conditions $i \in I_j$ and $j \in J_i$ are equivalent to each other, so that the maps are inverse one to another. 
\end{proof}

\subsection{Mixed decompositions of the simplex}
\label{subsecmixeddecomp}

Let us recall that $\Delta_{[m]}$ denotes the standard $m$-simplex whose vertices are labelled by the integers $0, \dots, m$. Every $m$-simplex whose vertices are totally ordered becomes canonically isomorphic to $\Delta_{[m]}$. Likewise, for every subset $J$ of $\{0, \dots , m\}$, we denote by $\Delta_J$ the face of $\Delta_{[m]}$ whose vertices belong to $J$.

Let then $I \in I(n,m)$ and $\alpha = (\alpha_0 , \dots , \alpha_n) \in \R_+^{n+1}$ be such that $\alpha_0 + \dots +  \alpha_n = 1$. We set 
$$\Delta_{I, \alpha} = \{ \alpha_0 x_0 + \dots +  \alpha_n x_n \in \Delta_{[m]} \, \vert \, \forall j \in \{ 0, \dots , n \} , \, x_j \in \Delta_{I_j} \}.$$
When all the $\alpha_j$'s equal $\frac{1}{n+1}$, this cell $\Delta_{I, \alpha}$ is thus the rescaled Minkowski sum $\Delta_{I_0} + \dots + \Delta_{I_n}$. 

Likewise, for every  $j \in \{ 0, \dots , n \} $, we denote by $T_{I_j}$ the basic tile $\Delta_{I_j} \setminus \Delta_{I_j \setminus \{ e_I (j) \}}$ with the convention that $\Delta_\emptyset = \emptyset$ and set
$$T_{I, \alpha} = \{ \alpha_0 x_0 + \dots +  \alpha_n x_n \in \Delta_{[m]} \, \vert \, \forall j \in \{ 0, \dots , n-1 \} , \, x_j \in T_{I_j}  \text{ and } x_n \in \Delta_{I_n} \}.$$

\begin{ex}
\label{exMinkowski}
1) If $m=n=2$ and $\alpha = (\frac{1}{3}, \frac{1}{3} , \frac{1}{3})$, then $I (2,2)$ consists of six staircases which, once labelled in the increasing order, are $I^1 = (\{ 0 \} , \{0 \} , \{ 0, 1, 2 \})$, $I^2 = (\{ 0 \} , \{0, 1 \} , \{  1, 2 \})$, $I^3 = (\{ 0 \} , \{0, 1, 2 \} , \{ 2 \})$, $I^4 = (\{ 0, 1 \} , \{1 \} , \{ 1, 2 \})$, $I^5 = (\{ 0, 1 \} , \{1,2 \} , \{ 2 \})$ and $I^6 = (\{ 0, 1, 2 \} , \{2 \} , \{ 2 \})$. The six cells $\big(\Delta_{I^N, \alpha} \big)_{N \in \{1, \dots , 6 \}}$ provide a mixed decomposition of the triangle $\Delta_{[2]}$ and the family $\big(T_{I^N, \alpha} \big)_{N \in \{1, \dots , 6 \}}$  provides a partition of the latter, depicted in Figure \ref{fig22}. \\

 \begin{figure}[h]
   \begin{center}
    \includegraphics[scale=1]{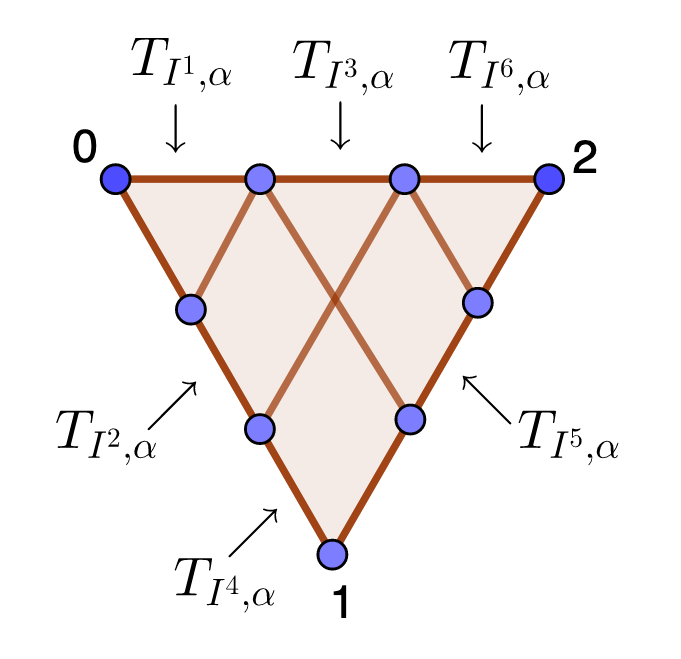}
    \caption{A mixed decomposition of the two-simplex.}
    \label{fig22}
      \end{center}
 \end{figure}

2) If $n=1$, $m=3$ and $\alpha = (\frac{1}{2}, \frac{1}{2})$, then $I(1,3)$ consists of four staircases which, once labelled in the increasing order, are $I^1 = (\{ 0 \} , \{ 0, 1, 2, 3 \})$, $I^2 = (\{ 0, 1 \} , \{  1, 2, 3 \})$, $I^3 = ( \{0, 1, 2 \} , \{ 2,3 \})$ and $I^4 = (\{ 0, 1, 2, 3 \} , \{3 \})$. The corresponding cells $\big(\Delta_{I^N, \alpha} \big)_{N \in \{1, \dots , 4 \}}$ are depicted in Figure \ref{fig13}
and provide a mixed decomposition of the simplex $\Delta_{[3]}$, while the family $\big(T_{I^N, \alpha} \big)_{N \in \{1, \dots , 4 \}}$  provides a partition of the latter, depicted in Figure \ref{fig13bis}. \\

 \begin{figure}[h]
   \begin{center}
    \includegraphics[scale=0.3]{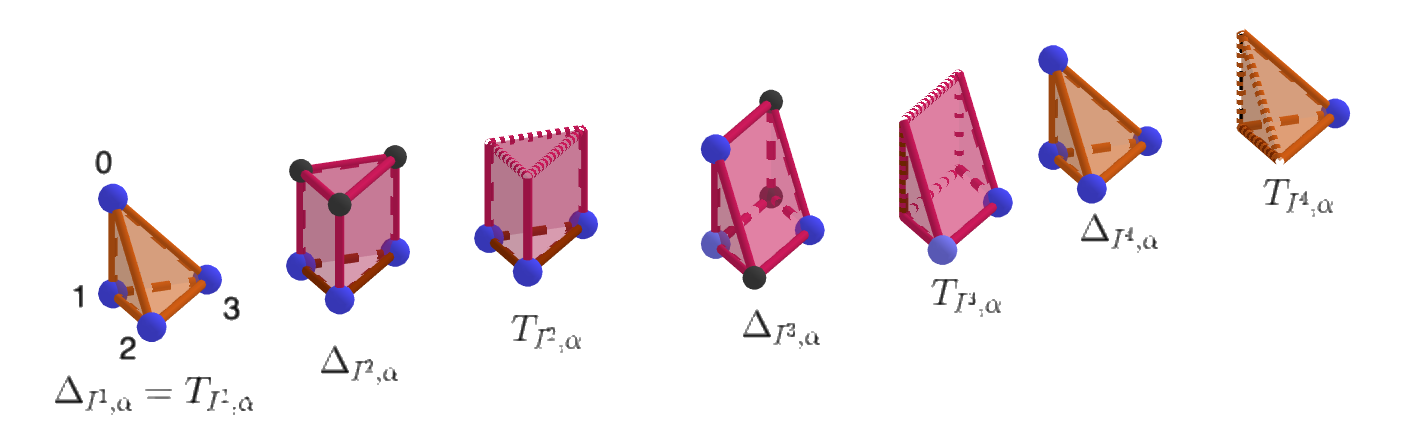}
    \caption{The four cells when $n=1$ and $m=3$.}
    \label{fig13}
      \end{center}
 \end{figure}

 \begin{figure}[h]
   \begin{center}
    \includegraphics[scale=0.3]{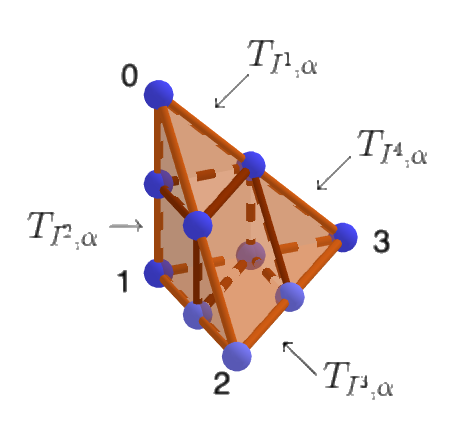}
    \caption{A mixed decomposition of the three-simplex.}
    \label{fig13bis}
      \end{center}
 \end{figure}

3) If $n=2$, $m=3$ and $\alpha = (\frac{1}{3}, \frac{1}{3} , \frac{1}{3})$, then $I(2,3)$ consists of ten staircases labelled in the increasing order by $(I^N)_{N \in \{1, \dots , 10 \}}$. The ten cells $\big(\Delta_{I^N, \alpha} \big)_{N \in \{1, \dots , 10 \}}$ provide a mixed decomposition of the simplex $\Delta_{[3]}$ and the family $\big(T_{I^N, \alpha} \big)_{N \in \{1, \dots , 10 \}}$  provides a partition of the latter, depicted in Figure \ref{fig23}. 
 \begin{figure}[h]
   \begin{center}
    \includegraphics[scale=0.3]{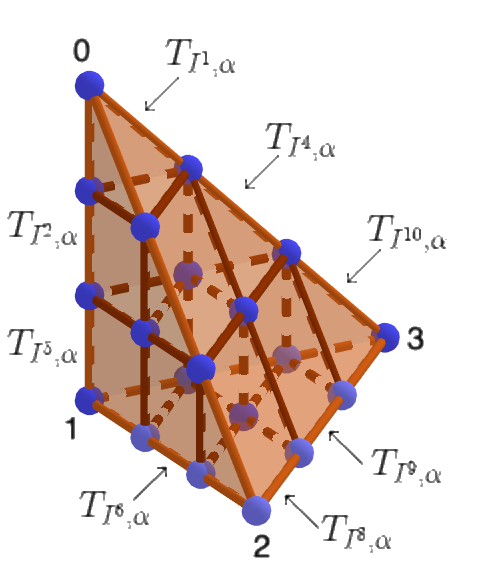}
    \caption{Another mixed decomposition of the three-simplex.}
    \label{fig23}
      \end{center}
 \end{figure}

\end{ex}
The phenomenon observed in Example \ref{exMinkowski} is general, compare \cite{San}.

\begin{theo}
\label{theoMinkowskisum}
Let $m,n$ be two non-negative integers and let $\alpha = (\alpha_0 , \dots , \alpha_n) \in (\R_+^*)^{n+1}$ be such that $\alpha_0 + \dots +  \alpha_n = 1$. Then, the Minkowski cells $\big(\Delta_{I, \alpha} \big)_{I \in I(n,m) }$ provide a mixed decomposition of the simplex $\Delta_{[m]}$ and the family $\big(T_{I, \alpha} \big)_{I \in I(n,m)}$  provides a partition of the latter. Moreover, if we label the staircases of $I(n,m)$ in increasing order by $(I^N)_{N \in \{1, \dots , N(n,m) \}}$, then for every $N \in \{1, \dots , N(n,m) \}$, the unions $\cup_{k=1}^N \Delta_{I^k, \alpha}$ and $\cup_{k=1}^N T_{I^k, \alpha}$ coincide and filtrate $\Delta_{[m]}$. Finally, the intersection of two cells $\Delta_{I, \alpha}$ and $\Delta_{I', \alpha}$, $I, I' \in I(n,m)$, is the face $\sum_{j=0}^n \alpha_j \Delta_{I_j \cap I'_j}$.
\end{theo}

\begin{proof}
Let us identify each point of the simplex $\Delta_{[m]}$ with its barycentric coordinates in the basis given by its vertices, so that $\Delta_{[m]} = \{ (\lambda_0 , \dots , \lambda_m) \in \R_+^{m+1} \, \vert \, \lambda_0 + \dots +  \lambda_m = 1 \}$. Then, for every $I = (I_j)_{j \in \{0, \dots , n\}} \in I(n,m)$,
\begin{eqnarray}
\label{eqndeltaIalpha}
\Delta_{I, \alpha} = \left\{ \lambda = (\lambda_i) \in \Delta_{[m]} \, \left| \, \forall j \in \{0, \dots , n\}, \sum_{i=0}^{e_I (j) - 1} \lambda_i \leq \sum_{l=0}^j \alpha_l \leq  \sum_{i=0}^{e_I (j)} \lambda_i \right. \right\}
\end{eqnarray}
Indeed, let $x \in \Delta_{I, \alpha}$, so that $x = \alpha_0 x_0 + \dots + \alpha_n x_n$ with $x_j \in \Delta_{I_j}$, $j \in \{0, \dots , n\}$, and let us denote the barycentric coordinates of $x_j$ by $(\lambda_i^j)_{i \in I_j}$. The barycentric coordinates $(\lambda_i)_{i \in \{ 0, \dots , m \}}$ of $x$ then satisfy, for every $i \in \{ 0, \dots , m \}$, $\lambda_i = \sum_{l=0}^n \alpha_l \lambda_i^l$. Let $j \in \{ 0, \dots , n \}$, we deduce that
\begin{eqnarray*}
\sum_{i=0}^{e_I (j)} \lambda_i  &=& \sum_{i=0}^{e_I (j)} \sum_{l=0}^{n} \alpha_l \lambda_i^l \geq  \sum_{l=0}^{j} \alpha_l \sum_{i=0}^{e_I (j)} \lambda_i^l =  \sum_{l=0}^{j} \alpha_l ,
\end{eqnarray*}
since $\{ 0 , \dots , e_I (j) \}$ contains $I_l$ if $l \leq j$. Likewise, 
$$\sum_{l=0}^{j} \alpha_l = \sum_{l=0}^{j} \alpha_l \big( \sum_{i=0}^{e_I (j)} \lambda_i^l \big) =  \sum_{i=0}^{e_I (j)} \big( \sum_{l=0}^{j} \alpha_l \lambda_i^l \big) \geq \sum_{i=0}^{e_I (j)-1} \big( \sum_{l=0}^{j} \alpha_l \lambda_i^l \big) = \sum_{i=0}^{e_I (j)-1} \lambda_i,$$
since $I_l \cap \{ 0 , \dots , e_I (j)-1 \} = \emptyset$ if $l > j$.

Conversely, if $\lambda = (\lambda_i)_{i \in \{ 0, \dots , m \}}$ satisfies, for every $j \in \{0, \dots , n\}$, $\sum_{i=0}^{e_I (j) - 1} \lambda_i \leq \sum_{l=0}^j \alpha_l \leq  \sum_{i=0}^{e_I (j)} \lambda_i $, we set
$\lambda_i^0 = \frac{\lambda_i}{\alpha_0}$ if $i < e_I (0)$, $\lambda_i^0 =0$ if $i > e_I (0)$ and $\lambda_{e_I (0)}^0 = 1 - \sum_{i=0}^{e_I (0) - 1} \lambda^0_i$. Then, by induction on $j$, we set
\begin{eqnarray}
\label{eqnlamndaik}
\lambda_i^j =  \frac{\lambda_i - \sum_{l=0}^{j-1} \alpha_l \lambda_i^l}{\alpha_j} \text{ if } i < e_I (j), \, \lambda_i^j =0 \text{ if } i > e_I (j) \text{ and } \lambda_{e_I (j)}^j = 1 - \sum_{i=0}^{e_I (j) - 1} \lambda^j_i.
\end{eqnarray}
These coefficients are all non-negative and if we denote by $x_j$ the point with barycentric coordinates $(\lambda_i^j)_{i \in I_j}$, we get $x = \alpha_0 x_0 + \dots + \alpha_n x_n$ by construction. The equality (\ref{eqndeltaIalpha}) is proved. 

Now, for every $N \in \{ 1, \dots , N(n,m) \}$, let us denote by $L_N$ the union $\cup_{k=1}^N \Delta_{I^k, \alpha}$. Let $\lambda = (\lambda_i)_{i \in \{ 0, \dots , m \}} \in \Delta_{[m]}$ and for every $j \in \{ 0, \dots , n-1 \}$, let $f(j)$ be the least integer $e$ such that $\sum_{i=0}^e \lambda_i \geq \sum_{l=0}^j \alpha_l$, so that $\sum_{l=0}^j \alpha_l > \sum_{i=0}^{f(j)-1} \lambda_i$. We set $f(n)=m$. Then, $\lambda \in \Delta_{I^f, \alpha}$ by (\ref{eqndeltaIalpha}) while if $I < I^f$, $\lambda \notin \Delta_{I, \alpha}$ by definition of the lexicographic order.  In particular, the cells $(\Delta_{I, \alpha})_{I \in I(n,m)}$ cover $\Delta_{[m]}$. Let $N \in \{ 1, \dots , N(n,m) \}$ be such that $I^f = I^N$, we deduce that $\Delta_{I^N, \alpha} \cap L_{N-1}= \{ \lambda = (\lambda_i)_{i \in \{ 0, \dots , m \}} \in \Delta_{I^N, \alpha} \, \vert \, \exists j \in \{ 0 , \dots , n-1 \}, \, \sum_{l=0}^j \alpha_l = \sum_{i=0}^{e_{I^N} (j)-1} \lambda_i  \}$. Let then $\lambda = (\lambda_i)_{i \in \{ 0, \dots , m \}} \in \Delta_{I^N, \alpha} \cap L_{N-1}$ and $ j \in \{ 0 , \dots , n-1 \}$ be such that $ \sum_{l=0}^j \alpha_l = \sum_{i=0}^{e_{I^N} (j)-1} \lambda_i $. Since $\alpha_j \neq 0$, this forces $e_{I^N} (j-1) < e_{I^N} (j)$.  Then, denoting by $x$ the point of $\Delta_{[m]}$ with barycentric coordinates $\lambda$ and writing it as before $x = \alpha_0 x_0 + \dots + \alpha_n x_n$, where $x_l$ has barycentric coordinates $(\lambda^l_i)_{i \in I^N_l}$, $l \in \{ 0, \dots , n\}$, this forces $\lambda^j_{e_{I^N} (j)} = 0$. Indeed, by (\ref{eqnlamndaik}), $\sum_{i=0}^{e_{I^N} (j)-1} \lambda_i^j = \frac{1}{\alpha_j} \big( \sum_{i=0}^{e_{I^N} (j)-1} \lambda_i - \sum_{l=0}^{j-1} \alpha_l (\sum_{i=0}^{e_{I^N} (j)-1} \lambda^l_i ) \big) = \frac{1}{\alpha_j} \big( \sum_{i=0}^{e_{I^N} (j)-1} \lambda_i - \sum_{l=0}^{j-1} \alpha_l \big) = \frac{\alpha_j}{\alpha_j} =1$. We conclude that $x_j$ belongs to the facet $\Delta_{I^N_j \setminus \{ e_{I^N} (j) \}}$ of $\Delta_{I^N_j}$. Conversely, if $x_j$ belongs to the facet $\Delta_{I^N_j \setminus \{ e_{I^N} (j) \}}$ of $\Delta_{I^N_j}$, then $e_{I^N} (j-1) < e_{I^N} (j)$ and the preceding computation shows that
$\sum_{i=0}^{e_{I^N} (j)-1} \lambda_i  = \sum_{l=0}^{j} \alpha_l $ since $\sum_{i=0}^{e_{I^N} (j)-1} \lambda_i^j  =1$. We deduce that for every $N \in \{ 2, \dots , N(n,m) \}$, $L_N \setminus L_{N-1}= T_{I^N, \alpha}$.

Let us finally prove that the intersection of two cells is a common face of them. Let then $I = (I_j)_{j \in \{ 0, \dots , n \}}$ and $I' = (I'_j)_{j \in \{ 0, \dots , n \}}$ be two staircases of $I(n,m)$. For every $j \in \{ 0, \dots , n \}$, let $x_j \in \Delta_{I_j} \cap \Delta_{I'_j} $. Then, $\sum_{j=0}^n \alpha_j x_j \in \Delta_{I, \alpha} \cap \Delta_{I', \alpha}$ and we have to prove that the intersection is reduced to this face. Let $x$ be a point in this intersection and let $\lambda = (\lambda_i)_{i \in \{ 0, \dots , m \}} $ be its barycentric coordinates. For every $j \in \{ 0, \dots , n \}$,
$\sum_{i=0}^{e_I (j) - 1} \lambda_i \leq \sum_{l=0}^j \alpha_l \leq  \sum_{i=0}^{e_I (j)} \lambda_i$ and $\sum_{i=0}^{e_{I'} (j) - 1} \lambda_i \leq \sum_{l=0}^j \alpha_l \leq  \sum_{i=0}^{e_{I'} (j)} \lambda_i $ by (\ref{eqndeltaIalpha}). We may write $x = \alpha_0 x_0 + \dots + \alpha_n x_n$ and $x = \alpha_0 x'_0 + \dots + \alpha_n x'_n$, where for every $j \in \{ 0, \dots , n \}$, $x_j \in \Delta_{I_j} $ and $x'_j \in \Delta_{I'_j} $ have barycentric coordinates given by (\ref{eqnlamndaik}). Let $ j \in \{ 0, \dots , n \}$. If $e_I (j) < e_{I'} (j)$ and vice versa if $e_{I'} (j) < e_I (j)$, we deduce that $\sum_{l=0}^j \alpha_l = \sum_{i=0}^{e_I (j)} \lambda_i  = \sum_{i=0}^{e_{I'} (j) - 1} \lambda_i $, so that $\lambda_i =0$ if  $e_I (j) <  i <e_{I'} (j)$ and $x_j'$ belongs to the facet
$\Delta_{I'_j \setminus \{ e_{I'} (j)  \}} $ of $\Delta_{I'_j} $ by what we proved above. By induction on $j$, the formula (\ref{eqnlamndaik}) defining the barycentric coordinates of $x_j$ and $x'_j$ give then the same numbers, so that $x_j=x'_j \in \Delta_{I_j} \cap \Delta_{I'_j} $ for every $j \in \{ 0, \dots , n \}$. Hence the result. 
\end{proof}

\subsection{Staircases triangulations}
\label{subsectriangulations}

The mixed decompositions given by Theorem \ref{theoMinkowskisum}  provide in fact primitive triangulations of the product of two simplices via the Cayley trick, see \cite{GKZ, HRS, San}, inducing the cartesian product structure of \cite{EilSteen}. Indeed, the Cayley trick makes it possible to switch from triangulations to mixed decompositions by associating to every maximal simplex of a primitively triangulated product of simplices $\Delta \times \Delta'$ its intersection with the fiber $\{ b \} \times \Delta'$, where
$b$ denotes the barycenter of $\Delta$. We are going to use this correspondence. 

For every staircase $I \in I(n,m)$, let $\Delta_I$ be the convex hull in $\Delta_{[n]} \times \Delta_{[m]}$ of the faces $(\{ j \} \times \Delta_{I_j})_{j \in \{ 0, \dots , n \}}$ and let $T_I$ be the convex hull of the tiles $(\{ j \} \times T_{I_j})_{j \in \{ 0, \dots , n-1 \}}$ and $\{ n \} \times \Delta_{I_n}$.

\begin{cor}
\label{corprimtriangulations}
For every non-negative integers $m,n$, the ordered collection of $(m+n)$-simplices $(\Delta_I)_{I \in I(n,m)}$ provides a shelled primitive triangulation of the product $\Delta_{[n]} \times \Delta_{[m]}$. The ordered collection of tiles $(T_I)_{I \in I(n,m)}$ provides the associated $h$-tiling.
\end{cor}
Corollary \ref{corprimtriangulations} corresponds to the case in Theorem \ref{theoproducttiles} where both tiles are critical of vanishing index. 

\begin{proof}
We prove the result by induction on $n$. If $n=0$, there is nothing to prove, the set $I(0,m)$ consists of a single staircase $I_0$ and $\Delta_{I_0}= T_{I_0}$ coincides with the simplex $\{ 0 \} \times \Delta_{[m]}$. Let us assume the result proven up to the rank $n-1$ and let us prove it for $n$. For every $j \in \{0, \dots , n\}$, the vertices of the facet $\Delta_{[n] \setminus \{ j \}}$ inherit a total order and by the induction hypothesis, the product $\Delta_{[n] \setminus \{ j \}}  \times \Delta_{[m]}$ inherits a triangulation with maximal simplices $(\Delta_{\tilde{I}})_{\tilde{I} \in I(n-1,m)}$.
If $I \in I(n,m)$, the intersection of $\Delta_I$ with $ \Delta_{[n] \setminus \{ j \}}  \times \Delta_{[m]}$ is a face of codimension $\# I_j - 1$ in this triangulation by definition, since it is included in a simplex $\Delta_{\tilde{I}}$ for some $\tilde{I} \in I(n-1,m)$ and even coincides with this simplex if $\# I_j = 1$. Likewise, for every $j \neq k \in \{0, \dots , n\}$, the vertices of the face $\Delta_{[n] \setminus \{ j , k \}}$ inherit a total order and the product $ \Delta_{[n] \setminus \{ j, k \}}  \times \Delta_{[m]}$ inherits by the induction hypothesis a triangulation which coincides from what we just saw with the trace of the triangulations of $\Delta_{[n] \setminus \{ j \}} \times \Delta_{[m]}$  and $\Delta_{[n] \setminus \{ k \}}  \times \Delta_{[m]}$. Hence, we get from the induction hypothesis a triangulation on $\partial \Delta_{[n]} \times \Delta_{[m]}$. Now, the interior points to $\Delta_{[n]}$ are determined by their barycentric coordinates $\alpha = (\alpha_0 , \dots , \alpha_n) \in (\R_+^*)^{n+1}$, $\alpha_0 + \dots + \alpha_n = 1$, in the affine basis given by its vertices. For every staircase $I \in I(n,m)$, the intersection $\Delta_I \cap (\{ \alpha \} \times \Delta_{[m]})$ coincides with $\Delta_{I, \alpha}$ by definition. From Theorem \ref{theoproducttiles} follows thus by induction that the collection of simplices $(\Delta_I)_{I \in I(n,m)}$ defines a primitive triangulation of the product $\Delta_{[n]} \times \Delta_{[m]}$ and that the intersection of two such simplices $\Delta_I \cap \Delta_J$ is the convex hull of the vertices of $(\{j\}\times \Delta_{I_j \cap J_j})_{j \in \{ 0, \dots , n \}}$ they have in common. 

Let us now label the straicases in the increasing order by $(I^N)_{N \in \{1 \dots , N(n,m)\}}$ so that the unions $\cup_{k=1}^N \Delta_{I^k}$, $N \in \{1 \dots , N(n,m)\}$, filtrate $\Delta_{[n]} \times \Delta_{[m]}$. Let $N \in \{1 \dots , N(n,m)\}$, we are going to prove that $\cup_{k=1}^N \Delta_{I^k} \setminus \cup_{k=1}^{N-1} \Delta_{I^k} $ is the tile $T_{I^N}$.
Let $I \in I(n,m)$ be such that $I < I^N$. By definition of the lexicographic order, there exists  $j \in \{0, \dots , n-1\}$ such that $\# I_j^N > 1$ and $e_I (j) < e_{I^N} (j)$, since $e_I$ is increasing.
If $e_I (j) = e_{I^N} (j) - 1$ and $e_I (l) = e_{I^N} (l)$ for $l \neq j$, then $\Delta_I$ contains all vertices of $\Delta_{I^N}$ but $(j, e_{I^N} (j))$ so that $\Delta_I \cap \Delta_{I^N}$ is the facet of $\Delta_{I^N}$ not containing $(j, e_{I^N} (j))$. Otherwise, $\Delta_I$ contains a subset of these vertices of $\Delta_{I^N}$. Since $T_{I^N}$ is the tile $\Delta_{I^N}$ deprived precisely of all those facets not containing $(j, e_{I^N} (j))$ for all $j \in \{0, \dots , n-1\}$ such that $\# I_j >1$, we deduce the result.
\end{proof}

Let us finally observe that the lexicographic order on the pairs $(j,i) \in \{ 0, \dots , n \} \times \{ 0, \dots , m\}$ induces a total order on the vertices of $\Delta_I$ for all $I \in I(n,m)$.
If $(j_1 , i_1)$ and $(j_2 , i_2)$ are two vertices of $\Delta_I$, then, by definition of staircases, $(j_1 , i_1) \leq (j_2 , i_2)$ with respect to this order iff $j_1 \leq j_2$ and $i_1 \leq i_2$. The product $\Delta_{[n]} \times \Delta_{[m]}$ equipped with the triangulation $(\Delta_I)_{I \in I(n,m)}$  given by Corollary \ref{corprimtriangulations} is thus the cartesian product of $\Delta_{[n]}$ and $ \Delta_{[m]}$ in the sense of Definition $II. 8.8$ of \cite{EilSteen}.

\subsection{The palindromic automorphism}
\label{subsecpalindautom}

Let ${\cal P}$ be the automorphism of $\Delta_{[n]} \times \Delta_{[m]}$ induced by reversing the total orders of the vertices of both simplices, so that for every $(j,i) \in \{0, \dots , n\} \times \{0, \dots , m\}$, ${\cal P} (j,i) = (n-j, m-i)$.

\begin{lemma}
\label{lemmapalindautom}
For every non-negative integers $m,n$, the automorphism ${\cal P}$ preserves the triangulation $(\Delta_I)_{I \in I(n,m)}$ of $\Delta_{[n]} \times \Delta_{[m]}$ and satisfies, for every $I \in I(n,m)$,
${\cal P} (\Delta_I) = \Delta_{\check{I}}$.
\end{lemma}

\begin{proof}
Let $I \in I(n,m)$, it is enough to prove that ${\cal P}$ maps the vertices of $\Delta_I$ on those of $\Delta_{\check{I}}$. Now, if $j \in \{0, \dots , n\} $ and if $i \in I_j$, ${\cal P} (j,i) = (n-j, m-i)$ by definition while
$\check{I}_{n-j} = \{ m - e_I (j) , \dots , m- b_I (j) \}$ by definition, so that $m-i \in \check{I}_{n-j}$ if $i \in I_j$.
\end{proof}
However, the automorphism ${\cal P}$ does not preserve the $h$-tiling $(T_I)_{I \in I(n,m)}$ of $\Delta_{[n]} \times \Delta_{[m]}$ in general, so that $({\cal P}(T_I))_{I \in I(n,m)}$ provides another shelling of the triangulation $(\Delta_I)_{I \in I(n,m)}$. We may also exchange the factors $ \Delta_{[n]}$ and $ \Delta_{[m]}$ to get from Corollary \ref{corprimtriangulations} a triangulation $(\Delta_J)_{J \in I(m,n)}$ on the product $\Delta_{[m]} \times \Delta_{[n]}$ together with a pair of shellings $(T_J)_{J \in I(m,n)}$ and $({\cal P}(T_J))_{J \in I(m,n)}$.

\begin{lemma}
\label{lemmaexchange}
The involution $E : (x,y) \in \Delta_{[n]} \times \Delta_{[m]} \mapsto (y, x) \in \Delta_{[m]} \times \Delta_{[n]} $ commutes with the action of ${\cal P}$ and defines an isomorphism between the simplicial complexes $(\Delta_I)_{I \in I(n,m)}$ and $(\Delta_J)_{J \in I(m,n)}$ which maps the shelling $(T_I)_{I \in I(n,m)}$ onto the shelling $({\cal P}(T_J))_{J \in I(m,n)}$. Moreover, it preserves the order on the vertices of all simplices $\Delta_I$, $I \in I(n,m)$.
\end{lemma}

\begin{proof}
Let us identify $x \in \Delta_{[n]}$ (resp. $y \in \Delta_{[m]}$) with its barycentric coordinates $(\alpha_j)_{j \in \{0, \dots , n\}}$ (resp. $(\lambda_i)_{i \in \{0, \dots , m\}}$) in the affine basis given
by the vertices of $\Delta_{[n]}$ (resp. $ \Delta_{[m]} $), so that $\alpha_j \geq 0$ (resp. $\lambda_i \geq 0$) and $\sum_{j=0}^n \alpha_j = 1$ (resp. $\sum_{i=0}^m \lambda_i = 1$ ). Then,
${\cal P} \circ E \big( (\alpha_j)_{j \in \{0, \dots , n\}} , (\lambda_i)_{i \in \{0, \dots , m\}} \big) = \big( (\lambda_{m-i})_{i \in \{0, \dots , m\}} , (\alpha_{n-j})_{j \in \{0, \dots , n\}} \big) = E \circ {\cal P} \big( (\alpha_j)_{j \in \{0, \dots , n\}} , (\lambda_i)_{i \in \{0, \dots , m\}} \big)$, hence the first part of the lemma. 

Let now $I \in I(n,m)$ and $\Delta_I$ be the associated simplex in the triangulated $\Delta_{[m]} \times \Delta_{[n]} $. Its vertices are the pairs $(j,i) \in \{0, \dots , n\} \times \{0, \dots , m\}$ such that $i \in I_j$. The image $E (\Delta_I)$ is thus the simplex of $\Delta_{[m]} \times \Delta_{[n]} $ with vertices $(i,j) \in \{0, \dots , m\} \times \{0, \dots , n\}$  such that $i \in I_j$. The conditions $i \in I_j$ and $j \in J_i$ are equivalent to each other, where $J$ is the image of $I$ under the involution $I \in I(n,m) \mapsto J \in J(m,n)$ given by Lemma \ref{lemmaJi}. We deduce that $E(\Delta_I) = \Delta_J$, so that $E$ defines a $\Z/2\Z$-equivariant simplicial isomorphism between the complexes $(\Delta_I)_{I \in I(n,m)}$ and $(\Delta_J)_{J \in I(m,n)}$. Moreover, if $(j_1 , i_1) \leq (j_2 , i_2)$ are vertices of $\Delta_I$, then $j_1 \leq j_2$ and $i_1 \leq i_2$ by definition of staircases, so that $(i_1 , j_1) \leq (i_2 , j_2)$ and
$E_{\vert \Delta_I} : \Delta_I \to \Delta_J$ preserves the order on the vertices.

Let us finally prove that ${\cal P} \circ E$ maps the tiling $(T_I)_{I \in I(n,m)}$ onto $(T_J)_{J \in I(m,n)}$. Let $I \in I(n,m)$. By definition, $T_I$ is the simplex $\Delta_I$ with vertices $\{ (j,i) \in \{0, \dots , n\} \times \{0, \dots , m\} \, \vert \, i \in I_j \}$ deprived, for every $j < n$ such that $b_I (j) \neq e_I (j)$, of the facet not containing the vertex $(j, e_I (j))$. From the preceding part, the involution $E$ maps $\Delta_I$ onto the simplex $\Delta_J$, $J \in J(m,n)$, and the vertices $\{ (j,e_I (j)) \in \{0, \dots , n-1\} \times \{1, \dots , m\} \, \vert \,  b_I (j) < e_I (j) \}$ onto the vertices $\{ (i,b_J (i)) \in \{1, \dots , m\} \times \{0, \dots , n-1\} \, \vert \,  b_J (i) < e_J (i) \}$. The involution ${\cal P} \circ E$ thus maps $\Delta_I$ onto $\Delta_{\widecheck{J}}$ and the vertices $\{ (j,e_I (j)) \in \{0, \dots , n-1\} \times \{1, \dots , m\} \, \vert \,  b_I (j) < e_I (j) \}$ onto the vertices $\{ (i,e_J (i)) \in \{0, \dots , m-1\} \times \{1, \dots , n\} \, \vert \,  b_J (i) < e_J (i) \}$, so that ${\cal P} \circ E (T_I) = T_{\widecheck{J}}$. Hence the result. 
\end{proof}

\section{Shellings on products of two tiles}
\label{secshellings}

\subsection{Preliminaries}

The proofs of Theorems \ref{theoproducttiles} and \ref{theoduality} are based on the following Propositions \ref{prop1} and \ref{prop2}.

\begin{prop}
\label{prop1}
Let $m,n \geq 0$, $I \in I(n,m)$ and $J = \{ b (J) , \dots , e (J) \}$ be an interval of $\{0, \dots , m\}$. Then, the intersection of the tile $T_I$ with $\Delta_{[n]} \times \Delta_{[m] \setminus J}$ is:
\begin{enumerate}
\item empty if there exists $j \in \{ 0 , \dots , n-1 \}$ such that $b_I (j) \neq e_I (j)$ and $e_I (j) \in J$.
\item the convex hull of the tiles $(\{l\} \times T_{I_l})_{l \in \{0, \dots , n-1\} \setminus \{j\}}$, $\{n\} \times \Delta_{I_n}$ and $\{j\} \times T_{I_j \setminus  J}$ if there exists
$j \in \{ 0 , \dots , n-1 \}$ such that $J \subset I_j \setminus \{ b_I (j) , e_I (j)\}$ and the convex hull of the tiles $(\{ j \} \times T_{I_j})_{j \in \{0, \dots , n-1\}} $ and $\{n\} \times \Delta_{I_n \setminus J}$ if $J \subset I_n \setminus \{ b_I (n)\}$.
\item the convex hull of the tiles $(\{l\} \times T_{I_l})_{l \in \{j+1, \dots , n-1\}}$, $\{n\} \times \Delta_{I_n}$ and $\{j\} \times T_{I_j \setminus J}$ if  $ b (J) = 0$ and there exists $j \in \{ 0 , \dots , n-1 \}$ such that $J \subset I_j \setminus \{ e_I (j) \}$ and restricts to $\{n\} \times \Delta_{I_n \setminus J}$ if  $ b (J) = 0$ and $J \subset I_n$. 
\end{enumerate}
\end{prop}

\begin{remark}
\label{remprop1}
In particular, the face $T_I \cap (\Delta_{[n]} \times \Delta_{[m] \setminus J})$ is empty in the case 1. of Proposition \ref{prop1}, of codimension $\# J$ in the case 2. and of codimension $\# J + j-1$ in the case 3.
\end{remark}

\begin{proof}
In the case 1., the intersection of $\Delta_{[n]} \times \Delta_{[m] \setminus J}$ with the simplex $\Delta_I$ is included in the facet of $\Delta_I$ which does not contain the vertex $(j, e_I (j))$. Since $j<n$, $T_I$ is already deprived of this facet by definition and we get 1. In the case 2., $\Delta_{[n]} \times \Delta_{[m] \setminus J}$  contains all the vertices of $\Delta_I$ except those with coordinates $(j,i)$ with $i \in J$. This product intersects thus $\Delta_I$ along a face of codimension $\# J$, convex hull of $\{j\} \times \Delta_{I_j \setminus J}$ and the simplices $(\{l\} \times \Delta_{I_l})_{l \neq j}$. 
We deduce part 2. after intersecting $T_I$ with this face. Finally, if $J \subset I_j$ and $ b (J) = 0$, then $I_l = \{ 0 \}$ if $l < j$ by definition. It follows that the intersection of $\Delta_{[n]} \times \Delta_{[m] \setminus J}$ with the simplex $\Delta_I$  is the convex hull of $\{j\} \times \Delta_{I_j \setminus J}$ and the faces $(\{l\} \times \Delta_{I_l})_{l > j}$. We deduce 3. after intersecting $T_I$ with this face.
\end{proof}

\begin{prop}
\label{prop2}
Let $m,n \geq 0$, $I \in I(n,m)$ and $J $ be a subset of $\{0, \dots , n\}$. Then, the intersection of the tile $T_I$ with $\Delta_{[n] \setminus J} \times \Delta_{[m]}$ is:
\begin{enumerate}
\item empty if there exists $j \in J \setminus \{ n \}$ such that $b_I (j) \neq e_I (j)$.
\item the convex hull of the tiles $(\{j\} \times T_{I_j})_{j \in \{0, \dots , n-1\} \setminus J}$ together with $\{n\} \times \Delta_{I_n} $ if $n \notin J$ otherwise.
\end{enumerate}
\end{prop}

\begin{remark}
\label{remprop2}
In particular, the face $T_I \cap (\Delta_{[n] \setminus J} \times \Delta_{[m]})$ is of codimension at least $\# J$ when nonempty.
\end{remark}

\begin{proof}
In the case $1.$, the intersection of $\Delta_{[n] \setminus J} \times \Delta_{[m]}$ with the simplex $\Delta_I$ is included in the facet of $\Delta_I$ which does not contain the vertex $(j , e_I (j))$.
Since $j<n$, $T_I$ has been deprived of this facet by definition and we get 1. The case 2. follows from the definition of $T_I$.
\end{proof}

\subsection{Proof of Theorem \ref{theoproducttiles}}
\label{subsecprooftheoproducttiles}

Let us denote by $\sigma_1, \sigma_2$ the underlying simplices of $T_1$ and $T_2$ and by $n$, $m$ their respective dimensions. Let us also choose total orders on the vertices of $\sigma_1$ and $\sigma_2$ in such a  way that if $T_1$ (resp. $T_2$) has been deprived of a Morse face, then the vertices of this face are the greatest of $\sigma_1$ (resp. $\sigma_2$). These orders induce isomorphisms between $\sigma_1$ (resp. $\sigma_2$) and $\Delta_{[n]}$ (resp. $\Delta_{[m]}$), so that $\sigma_1\times \sigma_2$ inherits the triangulation $(\Delta_I)_{I \in I(n,m)}$ given by Corollary \ref{corprimtriangulations} together with the shelling $(T_I)_{I \in I(n,m)}$. The palindromic automorphism ${\cal P}$ given by Lemma \ref{lemmapalindautom} induces then an isomorphism of the simplicial complex $\sigma_1\times \sigma_2$. We are going to prove the following alternative. Either, for every $I \in I(n,m)$, $T_I \cap (T_1 \times T_2)$ is a Morse tile and the collection $\big(T_I \cap (T_1 \times T_2)\big)_{I \in I(n,m)}$ provides a Morse shelling of $T_1 \times T_2$ as claimed by Theorem \ref{theoproducttiles} ; or, for every $I \in I(n,m)$, ${\cal P} (T_I) \cap (T_1 \times T_2)$ is a Morse tile and the collection $\big({\cal P}(T_I )\cap (T_1 \times T_2)\big)_{I \in I(n,m)}$ provides the Morse shelling claimed by Theorem \ref{theoproducttiles}. Hence, the trace with $T_1 \times T_2$ of one of the two $h$-tilings $(T_I)_{I \in I(n,m)}$ and $({\cal P}(T_I))_{I \in I(n,m)}$ of $\sigma_1\times \sigma_2$ provides the Morse shelling we are looking for. However, the total orders chosen on the vertices of $\sigma_1$ and $\sigma_2$ have to satisfy the previous condition for this result to hold true. To get the first part of Theorem \ref{theoproducttiles}, where $T_1$ and $T_2$ are basic, one of them being regular, and where we want a true shelling of $T_1 \times T_2$, that is using only basic tiles, we need an additional condition on these orders to be satisfied, namely Condition $h$ of \S \ref{subseccond}. In what follows, we are going to consider separately the case where $T_1$ and $T_2$ are both basic, the case where they are both Morse non-basic and the mixed case, one being basic, the other one being Morse non-basic. Let us recall that the lexicographic order on the pairs $(j,i) \in \{0, \dots , n\} \times \{0, \dots , m\}$ induces a total order on the vertices of $\Delta_I$ for every $I \in I(n,m)$, so that the total orders on the vertices of $\sigma_1$ and $\sigma_2$ induce total orders on the vertices of all simplices of the complex $\sigma_1\times \sigma_2$.

\subsubsection{The case of two basic tiles}
\label{subseccaseofbasictiles}

If both tiles are basic, the chosen order on the vertices of $\sigma_1$ (resp. $\sigma_2$) induces an isomorphism between $T_1$ (resp. $T_2$) and $\Delta_{[n]} \setminus \cup_{j \in J_1} \Delta_{[n] \setminus \{ j \}}$ (resp. $\Delta_{[m]} \setminus \cup_{i \in J_2} \Delta_{[m] \setminus \{ i \}}$), where $\# J_1$ (resp. $\# J_2$) is the order of $T_1$ (resp. $T_2$). If one of these tiles is in addition regular, in order to get a shelling of $T_1 \times T_2$ using only regular basic tiles, we need to assume that the total orders have been chosen in such a way that if $\{0,n\} \subset J_1$, then 
$\{0,m\} \cap J_2 \neq \emptyset$ and vice versa, that if $\{0,m\} \subset J_2$, then $\{0,n\} \cap J_1 \neq \emptyset$, see Condition $h$ of  \S \ref{subseccond}. In this case, applying the involution $T_1 \times T_2 \to T_2 \times T_1$ which exchanges the roles of $T_1$ and $T_2$, we may assume that if $n \in J_1$, then $m \in J_2$ and that if $0 \in J_2$, then $0 \in J_1$. If we do not assume this additional Condition $h$ on the total orders, then, even with this possibility to apply the involution $T_1 \times T_2 \to T_2 \times T_1$, we can only assume that one of these two properties holds true, either that if $n \in J_1$, then $m \in J_2$, or that if $0 \in J_2$, then $0 \in J_1$, but not both. The following proof then provides a Morse shelling on $T_1 \times T_2$, but not a true shelling in general, so that we need the additional Condition $h$ to get the first part of Theorem \ref{theoproducttiles}.

From Proposition \ref{prop1} we know that for every $I \in I(n,m)$ and every $i \in J_2 \setminus \{ 0 \}$, $T_I \cap \big( \Delta_{[n]}  \times \Delta_{[m] \setminus \{ i \}} \big)$ is either empty, or of codimension one in $T_I$, see Remark \ref{remprop1}. If $0 \in J_2$, then either $\# I_0 > 1$ and $T_I \cap \big( \Delta_{[n]}  \times \Delta_{\{1, \dots , m\}} \big)$ is of codimension one in $T_I$, or $\# I_0 = 1$ and this intersection is included in the facet $T_I \cap \big( \Delta_{\{1, \dots , n\}}  \times \Delta_{[m]} \big)$, so that $T_I \cap \big( \Delta_{[n]}  \times T_2 \big)$ is a basic tile since we assumed that $0 \in J_2$ implies $0 \in J_1$. Likewise, from Proposition \ref{prop2} we know that for every $I \in I(n,m)$ and every $j \in J_1 $, $T_I \cap \big( \Delta_{[n] \setminus \{ j \}}  \times \Delta_{[m]} \big)$ is either empty, or of codimension one in $T_I$, with the exception of $j=n$ if $\# I_n > 1$, but then the intersection is included in the facet $T_I \cap \big( \Delta_{[n]}  \times \Delta_{\{0, \dots , m-1\}} \big)$ since we also assumed that $n \in J_1$ implies $m \in J_2$. We then deduce that for every $I \in I(n,m)$, $T_I \cap ( T_1 \times T_2 )$ is a basic tile. Moreover, by definition, the order of $T_I$ equals $n - \# \{ j \in \{0, \dots , n-1\} \, \vert \, \#I_j = 1 \}$, so that the order of $T_I \cap ( T_1 \times T_2 )$ equals
\begin{eqnarray}
\label{eqnarrayorder}
n - \# \{ j \in \{0, \dots , n-1\} \setminus J_1 \, \vert \, \#I_j = 1 \} +  \# \{ i \in J_2 \, \vert \, i \notin b_I (\{1, \dots , n\}) \}
\end{eqnarray}
plus one in case $n \in J_1$ and $\# I_n = 1$. Since one of the tiles $T_1, T_2$ is regular, we deduce from Propositions \ref{prop1} and \ref{prop2} the upper and lower bounds $0 < \ord (T_I \cap ( T_1 \times T_2 )) \leq m+n$, for the term $(n - \# \{ j \in \{0, \dots , n-1\} \setminus J_1 \, \vert \, \#I_j = 1 \} )$ vanishes only if $I_n = \{0, \dots , m\} $ and in this case the second term in (\ref{eqnarrayorder}) does not vanish by hypothesis. Moreover, the last two terms equal $m+1$ only if $I_0 = \{0, \dots , m\} $ and $J_2 = \{0, \dots , m\} $ but in this case the first term in (\ref{eqnarrayorder}) is less than $n$ by hypothesis, for $J_1$ cannot be $\{0, \dots , n\} $. Hence, the shelled triangulation of $\Delta_{[n]}  \times \Delta_{[m] }$ given by Corollary \ref{corprimtriangulations} induces a shelled primitive triangulation of $T_1 \times T_2$ using only regular basic tiles. The first part of Theorem \ref{theoproducttiles} is proven. 
When Condition $h$ of \S \ref{subseccond} is not satisfied, we may still assume, applying the involution $E : T_1 \times T_2 \to T_2 \times T_1$ if necessary, that if $n \in J_1$, then $m \in J_2$. If $0 \in J_2$ but $0 \notin J_1$, then for every $I \in I(n,m)$, $\Delta_I \cap \big(  \Delta_{[n]}  \times \Delta_{\{1, \dots , m\}} \big)$ is a face of codimension greater than one as soon as $\# I_0 = 1$ and its vertices are the greatest of $\Delta_I$. The tile $T_I \cap ( T_1 \times T_2 ) $ is then Morse and we get a Morse shelling $\big( T_I \cap ( T_1 \times T_2 ) \big)_{I \in I(n,m)}$ which is tame. If $n \in J_1$ and $m \notin J_2$, we apply the involution $E$ to the tame Morse shelling $\big( T_J \cap ( T_2 \times T_1 ) \big)_{J \in I(m,n)}$ we just obtained to deduce from  Lemma \ref{lemmaexchange} that the collection  $\big( {\cal P}(T_I) \cap ( T_1 \times T_2 ) \big)_{I \in I(n,m)}$  defines a tame Morse shelling on $T_1 \times T_2$.

If $T_1$ and $T_2$ are open simplices, so that $J_1 = \{0, \dots , n\} $ and $J_2 = \{0, \dots , m\} $, then again,  by Propositions \ref{prop1} and \ref{prop2}, for every $I \in I(n,m)$, every $i \in J_2 \setminus \{ 0 \}$ and every $j \in J_1 \setminus \{ n \}$, $T_I \cap \big( \Delta_{[n]}  \times \Delta_{[m] \setminus \{ i \}} \big)$ and $T_I \cap \big( \Delta_{[n] \setminus \{ j \}}  \times \Delta_{[m]} \big)$ are of codimension one in $T_I$ when non-empty. Moreover, $T_I \cap \big( \Delta_{[n]}  \times \Delta_{\{1, \dots , m\}} \big)$ (resp. $T_I \cap \big( \Delta_{\{0, \dots , n-1\}}  \times \Delta_{[m]} \big)$) is either of codimension one in $T_I $ if $\# I_0 > 1$ (resp. $\# I_n = 1$), or included in the facet $T_I \cap \big( \Delta_{\{1, \dots , n\}}  \times \Delta_{[m]} \big)$ (resp. $T_I \cap \big( \Delta_{[n]}  \times \Delta_{\{0, \dots , m-1\}} \big)$) which is removed from $T_I$. Thus, the shelled triangulation of $\Delta_{[n]}  \times \Delta_{[m] }$ given by Corollary \ref{corprimtriangulations} induces a shelled primitive triangulation of $T_1 \times T_2$ also in this case and (\ref{eqnarrayorder}) remains valid. It remains to check that this $h$-tiling uses a unique critical tile, of index $m+n$. But as before, $\ord (T_I \cap ( T_1 \times T_2 )) >0$ for every $I \in I(n,m)$ while $\ord (T_I \cap ( T_1 \times T_2 )) = n+m+1$ forces $I_0 = \{0, \dots , m\}$. However now, $J_1 = \{0, \dots , n\}$ and $J_2 = \{0, \dots , m\}$, so that this staircase provides a critical tile of index $m+n$, namely an open simplex. 

If on the contrary $T_1$ is a closed simplex and $T_2$ an open one, then for every $I \in I(n,m)$, $T_I \setminus \cup_{i =1}^{m}   \big(  \Delta_{[n]} \times \Delta_{[m] \setminus \{ i \}}\big)$ is a basic tile of order $m$ by Proposition \ref{prop1} and $T_I \cap ( T_1 \times T_2 )$ is a basic tile of order $m+1$ if $\# I_0 \neq 1$ and a Morse tile of order $m$ otherwise, the Morse face being of codimension $\# b_I^{-1} (0)$. The only critical tile  of the Morse tiling of $T_1 \times T_2 $ is thus the tile $T_I \cap ( T_1 \times T_2 )$ with $I_n = \{0, \dots , m\} $ and $I_j = \{0\} $ if $j<n$. Its index equals $m$. Moreover, for all non-basic tile $T_I \cap ( T_1 \times T_2 )$ of the tiling, $I \in I(n,m)$, the vertices of the Morse face $T_I \cap \big( \Delta_{[n]}  \times \Delta_{\{1, \dots , m\}} \big)$  are the greatest of $\Delta_I$ with respect to the lexicographic order on $\{0, \dots , n\}  \times \{0, \dots , m\}$, so that the Morse shelling is tame. If $T_2$ is a closed simplex and $T_1$ an open one, we apply the exchange involution $E : T_2 \times T_1 \to T_1 \times T_2$ to the Morse shelling $\big( T_J \cap ( T_2 \times T_1 ) \big)_{J \in I(m,n)}$ we just obtained, to deduce from Lemma \ref{lemmaexchange} that the collection  $\big( {\cal P}(T_I) \cap ( T_1 \times T_2 ) \big)_{I \in I(n,m)}$  defines a tame Morse shelling on $T_1 \times T_2$.

\subsubsection{The case of two non-basic tiles}
\label{subseccaseofnonbasictiles}

If both $T_1$ and $T_2$ are Morse and not basic, then the chosen orders on the vertices of $\sigma_1$ and $\sigma_2$ induce isomorphisms between them and $\Delta_{[n]} \setminus \big( \cup_{j \in J_1}  \Delta_{[n] \setminus \{ j \}} \cup \Delta_{\{k_1, \dots , n\}} \big)$ and $\Delta_{[m]} \setminus \big( \cup_{i \in J_2}   \Delta_{[m] \setminus \{ i \}} \cup \Delta_{\{k_2, \dots , m\}} \big)$ respectively, where $1<k_1 \leq n$, $1<k_2\leq m$, $J_1 \subset \{ k_1 , \dots , n\}$ and $J_2 \subset \{ k_2 , \dots , m\}$. Applying the involution  $E : T_1 \times T_2 \to T_2 \times T_1$ which exchanges the roles of $T_1$ and $T_2$ if necessary, we may assume that if $n \in J_1$, then $m \in J_2$, see Lemma \ref{lemmaexchange}. In this case, we deduce from \S \ref{subseccaseofbasictiles} that for every $I \in I(n,m)$, $T_I \setminus \big( \cup_{j \in J_1}  (\Delta_{[n] \setminus \{ j \}} \times \Delta_{[m]})  \cup_{i \in J_2}  (\Delta_{[n]} \times \Delta_{[m] \setminus \{ i \}}) \big)$ is a basic tile of order $n - \# \{ j \in \{0, \dots , n-1\} \setminus J_1 \, \vert \, \#I_j = 1 \} +  \# \{ i \in J_2 \, \vert \, i \notin b_I (\{1, \dots , n\}) \}$ plus one in case $n \in J_1$ and $\# I_n = 1$, see (\ref{eqnarrayorder}). But from Proposition \ref{prop2} we know that $T_I \cap \big( \Delta_{\{k_1, \dots , n\}} \times \Delta_{[m]} \big)$ is empty if there exists $j \in \{0, \dots , k_1 - 1\}$ such that $\# I_j \neq 1$ while in the opposite case, this intersection contains $T_I \cap \big( \Delta_{[n]} \times \Delta_{\{k_2, \dots , m\}} \big)$ by Proposition \ref{prop1} and is of codimension $k_1$ in $T_I$. In this second case, $T_I \cap ( T_1 \times T_2 )$ is Morse and regular, since its order is not more than $m+n-k_1-k_2+2$. Moreover, the vertices of the Morse face are the greatest of $\Delta_I$. In the first case, we know from Proposition \ref{prop1} that $T_I \cap \big( \Delta_{[n]} \times \Delta_{\{k_2, \dots , m\}} \big)$ is empty if there exists $j \in \{0, \dots , n\}$ such that $\# I_j \neq 1$ and $e_I (j) < k_2$ and of codimension $k+k_2$ otherwise, where $k$ denotes the least integer such that $\# I_k \neq 1$. The tile $T_I \cap ( T_1 \times T_2 )$ is then again Morse and it is singular iff $k=k_1-1$, $I_k = \{0, \dots , m\}$, $J_1 = \{ k_1 , \dots , n\}$ and $J_2 = \{ k_2 , \dots , m\}$, since its order is bounded from above by $m+n-k_1-k_2+2$ with equality, when $k=k_1 - 1$, only in this case. Moreover, the vertices of the Morse face are the greatest of $\Delta_I$.
Hence, the product $T_1 \times T_2$ inherits the Morse shelled primitive triangulation $\big( T_I \cap ( T_1 \times T_2 ) \big)_{I \in I(n,m)}$ which uses a unique critical tile, of index $m+n-k_1-k_2+2$, iff $T_1$ and $T_2$ are critical of indices $n-k_1+1$ and $m-k_2+1$ respectively. 
In the case that $n \in J_1$, but $m \notin J_2$, we apply the involution $E$ to the tame Morse shelling $\big( T_J \cap ( T_2 \times T_1 ) \big)_{J \in I(m,n)}$ we just obtained and deduce from Lemma \ref{lemmaexchange} that the collection $\big( {\cal P}(T_I) \cap ( T_1 \times T_2 ) \big)_{I \in I(n,m)}$ defines a tame Morse shelling on $T_1 \times T_2$. Theorem \ref{theoproducttiles} is thus proven in the case of non basic tiles. 

\subsubsection{The mixed case of one basic and one non basic tiles}

In the mixed case, we may assume that $T_1$ is Morse and $T_2$ basic, applying the involution  $E : T_1 \times T_2 \to T_2 \times T_1$ if necessary.
Then the chosen orders on the vertices of the underlying simplices $\sigma_1$ and $\sigma_2$ induce isomorphisms between them and $\Delta_{[n]} \setminus \big( \cup_{j \in J_1}  \Delta_{[n] \setminus \{ j \}} \cup \Delta_{\{k_1, \dots , n\}} \big)$ and $\Delta_{[m]} \setminus  \cup_{i \in J_2} \Delta_{[m] \setminus \{ i \}}$ respectively, where $1<k_1 \leq n$, $J_1 \subset \{ k_1 , \dots , n\}$ and $J_2 \subset \{ 0 , \dots , m\}$. We first assume that if $n \in J_1$, then $m \in J_2$. In this case, we deduce from \S \ref{subseccaseofbasictiles} that for every $I \in I(n,m)$, $T_I \setminus \big( \cup_{j \in J_1}  (\Delta_{[n] \setminus \{ j \}} \times \Delta_{[m]})  \cup_{i \in J_2 \setminus \{ 0 \}}  (\Delta_{[n]} \times \Delta_{[m] \setminus \{ i \}}) \big)$ is a basic tile of order $n - \# \{ j \in \{0, \dots , n-1\} \setminus J_1 \, \vert \, \#I_j = 1 \} +  \# \{ i \in J_2 \setminus \{ 0 \} \, \vert \, i \notin b_I (\{1, \dots , n\}) \}$ plus one in case $n \in J_1$ and $\# I_n = 1$, see (\ref{eqnarrayorder}). But from Proposition \ref{prop2} we know that $T_I \cap \big( \Delta_{\{k_1, \dots , n\}} \times \Delta_{[m]} \big)$ is empty if there exists $j \in \{0, \dots , k_1 - 1\}$ such that $\# I_j \neq 1$ and in the opposite case, this intersection contains $T_I \cap \big( \Delta_{[n]} \times \Delta_{\{1, \dots , m\}} \big)$ by Proposition \ref{prop1} and is of codimension $k_1$ in $T_I$. In this second case, $T_I \cap ( T_1 \times T_2 )$ is Morse, since if $n \in J_1$ and $\# I_n \neq 1$, then $T_I \cap  \big( \Delta_{\{0, \dots , n-1\}} \times \Delta_{[m]} \big)$ is contained in the facet $T_I \cap  \big( \Delta_{[n]} \times \Delta_{\{0, \dots , m-1\}} \big)$ which is removed by hypothesis. This tile is moreover regular, since its order is not more than $m+n-k_1$, and the vertices of the Morse face are the greatest of $\Delta_I$. In the first case, we know from Proposition \ref{prop1} that $T_I \cap \big( \Delta_{[n]} \times \Delta_{\{1, \dots , m\}} \big)$ is of codimension $k+1$, $k$ being the least integer less than $k_1$ such that $\# I_k \neq 1$. The tile $T_I \cap ( T_1 \times T_2 )$ is then again Morse and it is singular iff $k=k_1-1$, $I_k = \{0, \dots , m\}$, $J_1 = \{ k_1 , \dots , n\}$ and $J_2 = \{ 0 , \dots , m\}$, since its order is bounded from above by $m+n-k$ with equality only in this case. Moreover, the vertices of the Morse face are the greatest of $\Delta_I$. Hence, the product $T_1 \times T_2$ inherits the tame Morse shelled primitive triangulation $\big( T_I \cap ( T_1 \times T_2 ) \big)_{I \in I(n,m)}$ which uses a unique critical tile, of index $m+n-k_1+1$, iff $T_1$ and $T_2$ are critical, of indices $n-k_1+1$ and $m$ respectively. 
If now $n \in J_1$, but $m \notin J_2$, we apply the exchange involution $E$. We deduce from (\ref{eqnarrayorder}) that for every $I \in I(m,n)$, $T_I \setminus \big( \cup_{i \in J_2} (\Delta_{[m] \setminus \{ i \}} \times  \Delta_{[n]}) \cup_{j \in J_1} (  \Delta_{[m] } \times \Delta_{[n] \setminus \{ j \}} ) \big)$ is a basic tile of order
$m - \# \{ i \in \{0, \dots , m-1\} \setminus J_2 \, \vert \, \#I_i = 1 \} +  \# \{ j \in J_1  \, \vert \, j \notin b_I (\{1, \dots , m\}) \}$. From Proposition \ref{prop1} we know that $T_I \cap ( \Delta_{[m]} \times \Delta_{\{k_1, \dots , n\}} )$ is empty if there exists $i \in \{0, \dots , m-1\}$ such that $b_I (i) \neq e_I (i) < k_1$ and otherwise, it is of codimension $k + k_1$ in $T_I$, where $k$ denotes the least integer such that $\# I_k \neq 1$. Moreover, the vertices of the Morse face are the greatest of $\Delta_I$. We deduce that $T_I \cap ( T_2 \times T_1 )$ is Morse and its order is bounded from above by $m+n+1-k-k_1$, with equality only if $J_1 = \{ k_1 , \dots , n\}$, $J_2 = \emptyset$ and $I_m = \{ 0 , \dots , n \}$. The collection  $\big( T_I \cap ( T_2 \times T_1 ) \big)_{I \in I(m,n)}$ thus defines a tame Morse shelling of  $T_2 \times T_1$ which contains a unique critical tile, of index $n+1-k_1$ iff $T_1$ and $T_2$ are both critical, of indices $n+1-k_1$ and $0$ respectively. By Lemma \ref{lemmaexchange}, the collection $\big( {\cal P}(T_I) \cap ( T_1 \times T_2 ) \big)_{I \in I(n,m)}$ then defines the tame Morse shelling on $T_1 \times T_2$ we are looking for. Hence the result.

\subsubsection{Remarks on the proof of Theorem \ref{theoproducttiles}}
\label{subseccond}

1) The shelling of $T_1 \times T_2$ is inherited from the shelling of $\sigma_1 \times \sigma_2$ given by Corollary \ref{corprimtriangulations}, via the choice of total orders on the vertices of the underlying simplices $\sigma_1$ and $\sigma_2$. When $T_1$ and $T_2$ are both basic, these orders fix isomorphisms between $T_1$, $T_2$ and $\Delta_{[n]} \setminus \cup_{j \in J_1} \Delta_{[n] \setminus \{ j \}}$, $\Delta_{[m]} \setminus \cup_{i \in J_2} \Delta_{[m] \setminus \{ i \}}$ respectively, where $n = \dim \sigma_1$, $m = \dim \sigma_2$, $J_1 \subset \{ 0 , \dots , n\}$ and $J_2 \subset \{ 0 , \dots , m\}$. In order to get $h$-tilings on $T_1 \times T_2$, we need to assume the\\

{\it Condition $h$} : If $\{0,n\} \subset J_1$, then  $\{0,m\} \cap J_2 \neq \emptyset$ and vice versa, if $\{0,m\} \subset J_2$, then $\{0,n\} \cap J_1 \neq \emptyset$.\\

Indeed, for example, neither the shelling $(T_I )_{I \in I(n,m)}$, nor the shelling $\big({\cal P}(T_I )\big)_{I \in I(n,m)}$ induces an $h$-tiling on $\Delta_{[n]} \setminus \big( (\Delta_{\{0, \dots , n-1\}}  \times \Delta_{[m]} )  \cup( \Delta_{\{1, \dots , n\} }  \times \Delta_{[m]} ) \big) $ in general. \\

2) Likewise, when $T_1$ or $T_2$ is not basic, we need to assume  the\\

{\it Condition $M$} :  If $T_1$ (resp. $T_2$) is a Morse tile which is not basic, then the vertices of its Morse face are the greatest among those of $\sigma_1$ (resp. of $\sigma_2$).\\

Indeed, for example, neither the shelling $(T_I )_{I \in I(n,m)}$, nor the shelling $\big({\cal P}(T_I )\big)_{I \in I(n,m)}$ induces a Morse tiling on $\big( \Delta_{[n]} \setminus (\Delta_{\{0, \dots , n-1\}} \cup \Delta_{\{n-1, n\} } ) \big) \times \big( \Delta_{[m]}  \setminus ( \Delta_{\{1, \dots , m\}} \cup \Delta_{\{0, 1\} } ) \big)$ in general. \\

3) Finally, even if these Conditions $h$ or $M$ are satisfied, the shelling $(T_I )_{I \in I(n,m)}$ of $\Delta_{[n]} \times \Delta_{[m]}$ given by Corollary \ref{corprimtriangulations} need not induce a shelling on $T_1 \times T_2$, it has been sometimes necessary to apply the palindromic isomorphism ${\cal P}$ to $(T_I )_{I \in I(n,m)}$, which amounts to reverse the total orders of the vertices of $\sigma_1$ and $\sigma_2$. The latter does not affect the triangulation $(\Delta_I )_{I \in I(n,m)}$ given by Corollary \ref{corprimtriangulations}.

\subsection{Proof of Theorem \ref{theoduality}}
\label{subsecpftheoduality}

Let us first prove that we may assume the tiles to be basic. 

\begin{prop}
\label{propMorsetobasic}
Let $T_1, T_2$ be two Morse tiles with underlying basic tiles $T'_1, T'_2$. Then, the Morse shellings of $T_1 \times T_2$ and $T'_1 \times T'_2$ given by Theorem \ref{theoproducttiles} have same $h$-vector.
\end{prop}

\begin{proof}
If $T_1$ (resp. $T_2$) is a Morse tile which is not basic, we denote by $\sigma_1$ (resp. $\sigma_2$) its underlying simplex and by $\tau_1$ (resp. $\tau_2$) its Morse face. The Morse shelling given by Theorem \ref{theoproducttiles} is inherited from a total order on the vertices of $\sigma_1$ (resp. $\sigma_2$)  such that the ones of $\tau_1$ (resp. $\tau_2$) are the greatest, see Condition $M$ in \S \ref{subseccond}.
The product  $\sigma_1 \times \sigma_2$ inherits then a triangulation $(\Delta_I )_{I \in I(n,m)}$ and a shelling $(T_I )_{I \in I(n,m)}$ given by Corollary \ref{corprimtriangulations}. We know from Proposition \ref{prop2}
(resp. Proposition \ref{prop1}) that for every $I \in I(n,m)$, the intersection of $T_I$  with $\tau_1 \times \sigma_2$ (resp. $\sigma_1 \times \tau_2$) is either empty, or of codimension at least equal to the one of $\tau_1$ in $\sigma_1$ (resp. $\tau_2$ in $\sigma_2$). This intersection thus does not contribute to the order of $T_I \cap (T_1 \times T_2)$.
Thus, if the tame Morse shellling of $T_1 \times T_2 $ given by Theorem \ref{theoproducttiles} is $\big( T_I \cap ( T_1 \times T_2 ) \big)_{I \in I(n,m)}$, then the Morse shelling of $T'_1 \times T'_2 $ given by Theorem \ref{theoproducttiles} is $\big( T_I \cap ( T'_1 \times T'_2 ) \big)_{I \in I(n,m)}$ and has the same $h$-vector. Otherwise, Theorem \ref{theoproducttiles} provides the Morse shellling $\big( {\cal P}(T_I) \cap ( T_1 \times T_2 ) \big)_{I \in I(n,m)}$, which is the image of  $\big( T_J \cap ( T_2 \times T_1 ) \big)_{J \in I(m,n)}$,  under the exchange involution by Lemma \ref{lemmaexchange}, so that it provides the Morse shellling $\big( {\cal P}(T_I) \cap ( T'_1 \times T'_2 ) \big)_{I \in I(n,m)}$ on $T'_1 \times T'_2 $ as well, which has same $h$-vector as $\big( {\cal P}(T_I) \cap ( T_1 \times T_2 ) \big)_{I \in I(n,m)}$. Hence the result. 
\end{proof}

Let us now prove that all Morse tilings of a product $T_1 \times T_2$ given by Theorem \ref{theoproducttiles} have same $h$-vector. By Proposition \ref{propMorsetobasic} we may assume the tiles to be basic. 

\begin{prop}
\label{propsamemvector}
Let $T_1, T_2$ be two basic tiles. Then, all Morse shellings of  $T_1 \times T_2$ given by Theorem \ref{theoproducttiles} have same $h$-vector.
\end{prop}

\begin{proof}
We may assume that $T_1 = \Delta_{[n]} \setminus \cup_{j \in J_1} \Delta_{[n] \setminus \{ j \}}$ and $T_2 = \Delta_{[m]} \setminus \cup_{i \in J_2} \Delta_{[m] \setminus \{ i \}}$,
where $J_1 \subset \{0, \dots , n\}$ and $J_2 \subset \{0, \dots , m\}$. We have to assume in addition that if $n \in J_1$, then $m \in J_2$. This can be done after possibly applying the involution $E : T_1 \times T_2 \to T_2 \times T_1$ which exchanges the roles of $T_1$ and $T_2$. The tiles of the Morse tiling given by Theorem \ref{theoproducttiles} are then the $\big( T_I \cap ( T_1 \times T_2 ) \big)_{I \in I(n,m)}$ and we are going to prove that for every $k \in \{ 0, \dots , m+n+1 \}$, the number of tiles of order $k$ in this collection only depends on $\# J_1$ and $\# J_2$. 
This implies the result, for applying the exchange involution $E$ or not does not affect the $h$-vector as well. We proceed by induction on the dimensions $n,m$. For every $j_1 \in J_1 \setminus \{0,n\}$ (resp. $i_2 \in J_2 \setminus \{0,m\}$), let us compare the $h$-vectors of the tilings of $T_1 \times T_2$ and $T'_1 \times T_2$ (resp. $T_1 \times T'_2$),
where $T'_1 = \Delta_{[n]} \setminus \cup_{j \in J_1\setminus \{j_1\}} \Delta_{[n] \setminus \{ j \}}$ (resp. $T'_2 = \Delta_{[m]} \setminus \cup_{i \in J_2\setminus \{i_2\}} \Delta_{[m] \setminus \{ i \}}$). Let $I \in I(n,m)$. From Proposition \ref{prop2} (resp. Proposition \ref{prop1}) we know that the intersection of $\Delta_{[n] \setminus \{ j_1\}} \times \Delta_{[m]} $
(resp. $\Delta_{[n] } \times \Delta_{[m] \setminus \{ i_2 \}}$) with $T_I$ is empty if $\# I_{j_1} >1$ (resp. if $i_2 \in b_I (\{1, \dots , n\})$) and is of codimension one otherwise. In the first case,
$j_1$ (resp. $i_2$) does not contribute to the order of $T_I \cap (T_1 \times T_2)$, so that this order coincides with the one of $T_I \cap (T'_1 \times T_2)$ (resp. $T_I \cap (T_1 \times T'_2)$) and in the second case, its contribution equals one, so that $\ord (T_I \cap (T_1 \times T_2)) = \ord (T_I \cap (T'_1 \times T_2)) + 1$ (resp. $\ord (T_I \cap (T_1 \times T_2)) = \ord (T_I \cap (T_1 \times T'_2)) + 1$). The staircases $I \in I(n,m)$ for which $\# I_{j_1} =1$ (resp. $i_2 \notin b_I (\{1, \dots , n\})$) are in bijective correspondence with the staircases $\tilde{I} \in I(n-1,m)$ (resp. $\tilde{I} \in I(n,m-1)$), this correspondence $\for_{j_1}$ (resp. $\for_{i_2}$) being induced by the inclusion $\{0, \dots , n\} \setminus \{ j_1\} \to \{0, \dots , n\}$ (resp. $\{0, \dots , m\} \setminus \{ i_2 \} \to \{0, \dots , m\}$). But if $\tilde{I} = \for_{j_1} (I)$ (resp. $\tilde{I} = \for_{i_2} (I)$)  with $\# I_{j_1} =1$
(resp. $i_2 \notin b_I (\{1, \dots , n\})$) and if $\widetilde{T}_1 = \Delta_{[n] \setminus \{ j_1\}} \setminus \cup_{j \in J_1\setminus \{ j_1\}} \Delta_{[n] \setminus \{ j,j_1 \}}$ (resp. $\widetilde{T}_2 = \Delta_{[m] \setminus \{ i_2 \}} \setminus \cup_{i \in J_2 \setminus \{ i_2 \}} \Delta_{[m] \setminus \{ i,i_2 \}}$), then $\ord (T_I \cap (T'_1 \times T_2)) = \ord (T_{\tilde{I}} \cap (\widetilde{T}_1 \times T_2))$ (resp. $\ord (T_I \cap (T_1 \times T'_2)) = \ord (T_{\tilde{I}} \cap (T_1 \times \widetilde{T}_2))$) by Proposition \ref{prop2} (resp. Proposition \ref{prop1}). We deduce the relation
\begin{eqnarray}
\label{eqnarraymtilde}
h(T_1 \times T_2) = h(T'_1 \times T_2)  - h(\widetilde{T}_1 \times T_2)^{[0]} + h(\widetilde{T}_1 \times T_2)^{[1]},
\end{eqnarray}
where for every $v=(v_0, \dots , v_{m+n}) \in \Z^{m+n+1}$, $v^{[0]} =(v_0, \dots , v_{m+n}, 0) \in \Z^{m+n+2}$ and $v^{[1]} =(0,v_0, \dots , v_{m+n}) \in \Z^{m+n+2}$. Likewise, 
\begin{eqnarray}
\label{eqnarraymtilde2}
h(T_1 \times T_2) = h(T_1 \times T'_2)  - h(T_1 \times \widetilde{T}_2)^{[0]} + h(T_1 \times \widetilde{T}_2)^{[1]}.
\end{eqnarray}
By deleting one after the other the elements of $J_1 \setminus \{0,n\}$ and  $J_2 \setminus \{0,m\}$, we then express the $h$-vector of $T_1 \times T_2$ in terms of the $h$-vector of product of tiles of dimensions $\leq n$ and $\leq m$ having lower order, and this expression does not depend on the specific position of the elements in $J_1 \setminus \{0,n\}$ and  $J_2 \setminus \{0,m\}$, it only depends on the number of such elements. We can likewise delete $n \in J_1$ if $m \in J_2$ and $0 \in J_2$ if $0 \in J_1$, the formula (\ref{eqnarraymtilde}) being valid in this case and we may also delete $0 \in J_1$ if $0 \notin J_2$  and $m \in J_2$ if $n \notin J_1$, in the same way as before. The only delicate case is the case where $0 \in J_2$ but $0 \notin J_1$ since we have applied the exchange involution $T_1 \times T_2 \to T_2 \times T_1$ in order to make sure that $m \in J_2$ if $n \in J_1$. In such a case, the tiling of $T_1 \times T_2$ is Morse and not an $h$-tiling. Let $I \in I(n,m)$. The intersection of $\Delta_{[n]} \times \Delta_{\{1, \dots , m\}}$ with $T_I$ is no longer empty if $\#I_0 = 1$, but it is of codimension greater than one in $T_I$, so that it is a Morse face which does not contribute to the order of $T_I \cap (T_1 \times T_2)$. If $\#I_0 >1$, it is of codimension one and contributes as one to the order of $T_I \cap (T_1 \times T_2)$. If we delete $0 \in J_2$, we thus again get, in the same way, the formula (\ref{eqnarraymtilde2}), even if the natures of the tilings of $T_1 \times T_2$, $T_1 \times T'_2$ and $T_1 \times \widetilde{T}_2$ now differ. We thus deduce by applying inductively finitely many times (\ref{eqnarraymtilde}), (\ref{eqnarraymtilde2}) an expression of $h( T_1 \times T_2)$ in terms of the $h$-vectors of the product of closed simplices of dimensions $\leq n$ and $\leq m$ and this expression only depends on the cardinalities of $J_1$ and $J_2$. Moreover, the $h$-vectors of the tilings 
$(T_I)_{I \in I(n,m)}$ or $({\cal P}(T_I))_{I \in I(n,m)}$ given by Corollary \ref{corprimtriangulations} coincide and only depend on the dimension $n,m$ of the closed simplices. The result follows.
\end{proof}

\begin{remark}
The formulas (\ref{eqnarraymtilde}), (\ref{eqnarraymtilde2}) make it possible to compute by induction the $h$-vector given by Theorem \ref{theoproducttiles} in terms of the $h$-vectors given by Corollary \ref{corprimtriangulations}. Moreover, the latter can be computed using a similar induction or by computing the face vector of the cartesian product of two simplices, but we do not detail these computations here. 
\end{remark}

It remains to prove the formula given in Theorem \ref{theoduality} and thanks to Proposition \ref{propMorsetobasic}, it is enough to prove it for basic tiles. 

\begin{prop}
\label{propduality}
Let $T_1, T_2$ be two basic tiles and $\widecheck{T}_1, \widecheck{T}_2$ be their dual ones. Then, all Morse shellings of  $T_1 \times T_2$ and $\widecheck{T}_1 \times \widecheck{T}_2$ given by Theorem \ref{theoproducttiles} satisfy $h(\widecheck{T}_1 \times \widecheck{T}_2) = \widecheck{h}(T_1 \times T_2)$.
\end{prop}

\begin{proof}
Let $n,m$ (resp. $k,l$) be the dimensions (resp. orders) of $T_1$ and $T_2$. Let us choose total orders on the vertices of  $T_1, T_2$ (resp. $\widecheck{T}_1, \widecheck{T}_2$) in order to obtain isomorphisms with
$\Delta_{[n]} \setminus \cup_{j=0}^{k-1} \Delta_{[n] \setminus \{ j \}}$ and $ \Delta_{[m]} \setminus \cup_{i=0}^{l-1} \Delta_{[m] \setminus \{ i \}}$
(resp. $\Delta_{[n]} \setminus \cup_{j=0}^{n-k} \Delta_{[n] \setminus \{ j \}}$ and $ \Delta_{[m]} \setminus \cup_{i=0}^{m-l} \Delta_{[m] \setminus \{ i \}}$).
From (\ref{eqnarrayorder}), we know that for every $I \in I(n,m)$, 
$$\ord \big(T_I \cap ( T_1 \times T_2 )\big) = n - \# \{ j \in \{k, \dots , n-1\} \, \vert \, \#I_j = 1 \} +  \# \{ i \in \{0, \dots , l-1\} \, \vert \, i \notin b_I (\{1, \dots , n\}) \}$$
plus one in case $k=n+1$ and $\# I_n = 1$. Likewise, by definition of $\check{I}$ and (\ref{eqnarrayorder}), 
$$\ord \big(T_{\check{I}} \cap ( \widecheck{T}_1 \times \widecheck{T}_2 )\big) = n - \# \{ j \in \{1, \dots , k-1\}  \, \vert \, \#I_j = 1 \} +  \# \{ i \in \{l, \dots , m\} \, \vert \, i \notin e_I (\{0, \dots , n-1\}) \}$$
plus one in case $k=0$ and $\# I_0 = 1$. By adding these quantities, whatever the value of $k$ is, we get
$$\begin{array}{l}
\ord \big(T_I \cap ( T_1 \times T_2 )\big)  + \ord \big(T_{\check{I}} \cap ( \widecheck{T}_1 \times \widecheck{T}_2 )\big) \\
= 2n +m+1 -  \# \{ j \in \{1, \dots , n-1\}  \, \vert \, \#I_j = 1 \} -\# e_I (\{0, \dots , n-1\}), \\
\end{array}$$
since $e_I (j) = b_I (j+1)$ for every $j \in \{0, \dots , n-1\} $ by definition, see \S \ref{subsecstraircases}. Moreover, since $e_I$ is increasing, $ \# e_I (\{0, \dots , n-1\}) = n - \# \{ j \in \{1, \dots , n-1\}  \, \vert \, \#I_j = 1 \} $, so that $\ord \big(T_I \cap ( T_1 \times T_2 )\big)  + \ord \big(T_{\check{I}} \cap ( \widecheck{T}_1 \times \widecheck{T}_2 )\big) = m+n+1$.
\end{proof}

Theorem \ref{theoproducttiles} now follows from Propositions \ref{propMorsetobasic}, \ref{propsamemvector} and \ref{propduality}.

\subsection{Proof of Corollary \ref{corisomtiles} and further examples}
\label{subsecexamples}

1) Corollary \ref{corhandle} is a special case of Theorem \ref{theoproducttiles} which produces Morse shellings on handles of any dimension and index. Figure \ref{figexhandles} provides some examples of such shellings, depicted using the associated mixed decompositions of the simplex $ \Delta_{[m]} $, see \S \ref{subsecmixeddecomp}.

 \begin{figure}[h]
   \begin{center}
    \includegraphics[scale=1]{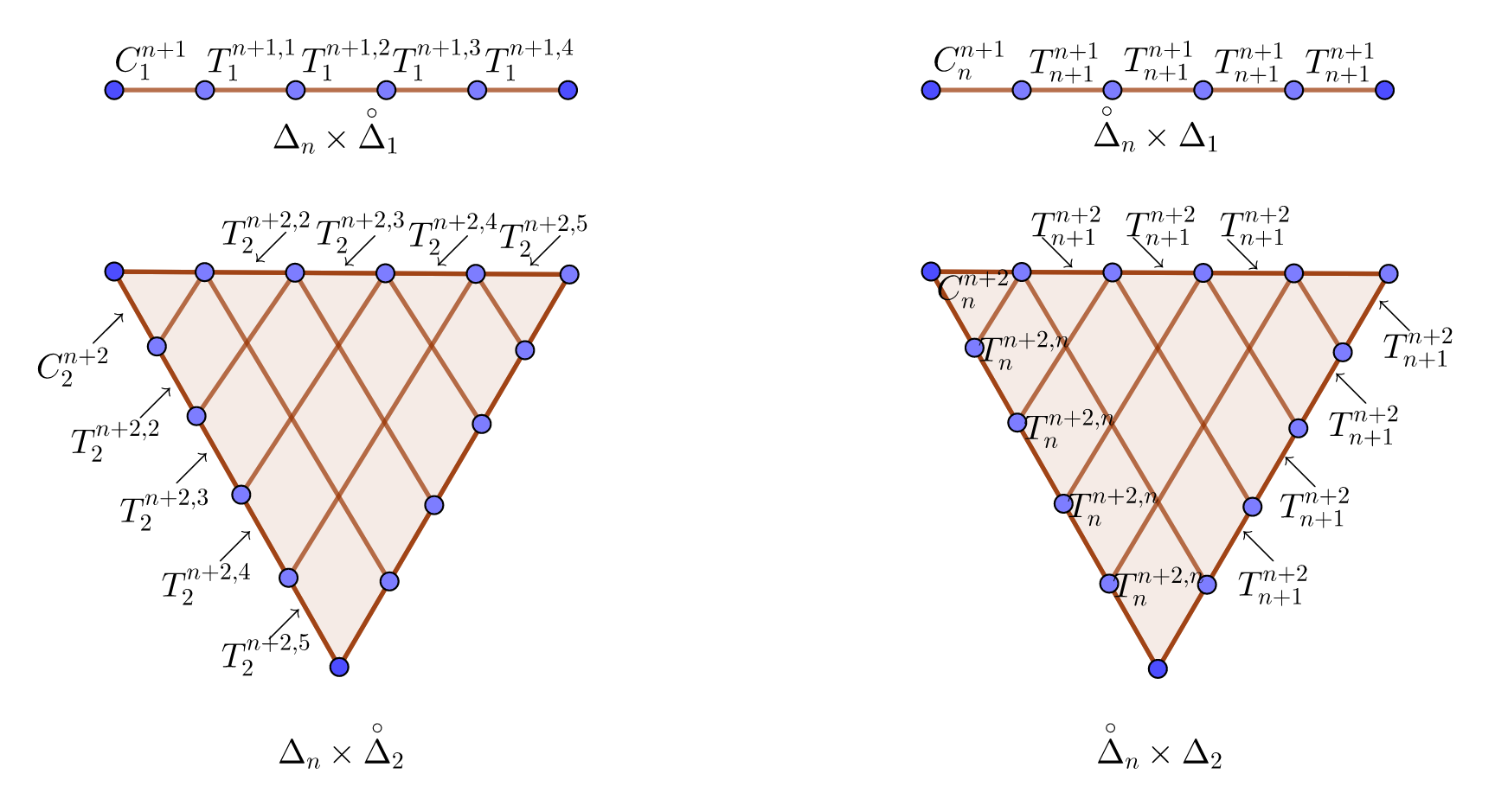}
    \caption{Morse shellings on handles.}
    \label{figexhandles}
      \end{center}
 \end{figure}

In general, we may check that the Morse shelling of the handle $\stackrel{\circ}{\Delta}_n \times \Delta_{m}$ given by Theorem \ref{theoproducttiles} uses one critical tile of index $n$, ${m+n-1 \choose n-1}$ basic tiles of order $n+1$ and for every $l \in \{ 2 , \dots , m\}$, ${m+n-l \choose n-1}$ Morse tiles isomorphic to $T_n^{m+n,m+n-l}$. Likewise, the handle $ \Delta_{n} \times \stackrel{\circ}{\Delta}_m$ is tiled by one critical tile of index $m$, ${m+n-1 \choose n}$ basic tiles of order $m+1$ and for every $l \in \{ m , \dots , m+n-2\}$, ${l\choose l+1-m}$ Morse tiles isomorphic to $T_m^{m+n,l}$. We do not detail these computations here. \\

2) Corollary \ref{corisomtiles} provides other special cases of shellings given by Theorem \ref{theoproducttiles}, namely $h$-tilings whose tiles are all isomorphic to each other.  Let us prove now this corollary. 

\begin{proof}[Proof of Corollary \ref{corisomtiles}]
Let $T_1 = \Delta_{[n]} \setminus \cup_{j=0}^{n-1} \Delta_{[n] \setminus \{ j \}}$ and $T_2 =  \Delta_{[m]}$. The shelled triangulation of $T_1 \times T_2$ given by Theorem \ref{theoproducttiles} uses the tiles $\big( T_I \cap (T_1 \times T_2)\big)_{I \in I(n,m)}$. From (\ref{eqnarrayorder}), we know that for every $I \in I(n,m)$, $\ord \big(T_I \cap ( T_1 \times T_2 )\big) = n$. Moreover, these tiles are all basic, so that the first part of Corollary \ref{corisomtiles} is proved. Likewise, let $T'_1 = \Delta_{[n]} \setminus \Delta_{\{1, \dots , n\}}$ and $T'_2 =  \Delta_{[m]} \setminus \cup_{i=0}^{m} \Delta_{[m] \setminus \{ i \}}$. From (\ref{eqnarrayorder}), we know that for every $I \in I(n,m)$,

\begin{eqnarray*}
\ord \big(T_I \cap ( T'_1 \times T'_2 )\big) &=&  n - \# \{ j \in \{1, \dots , n-1\} \, \vert \, \#I_j = 1 \} +  \# \{ i \in [m] \, \vert \, i \notin b_I (\{1, \dots , n\}) \}\\
&=& \# e_I (\{0, \dots , n-1\}) +m+1 - \# e_I (\{0, \dots , n-1\}) \\
&= &m+1.\\
\end{eqnarray*}
Again, all these tiles are basic, hence the result. 
\end{proof}

In the same way, $\Delta_{[n]} \times T_m^m$ (resp. $\stackrel{\circ}{\Delta}_n \times T_1^m$) inherits a shelled triangulation whose tiles are all isomorphic to each other, of order $m$ (resp. of order $n+1$).\\

3) We have observed in \S \ref{subsecpftheoduality} that the symmetry given by Theorem \ref{theoduality} is induced by the involution $I \in I(n,m) \mapsto \check{I} \in I(n,m)$. This symmetry appears on the examples given by Figures \ref{figexduality1}, \ref{figexduality2}, \ref{figexduality3} and \ref{figexduality4}.

 \begin{figure}[h]
   \begin{center}
    \includegraphics[scale=1]{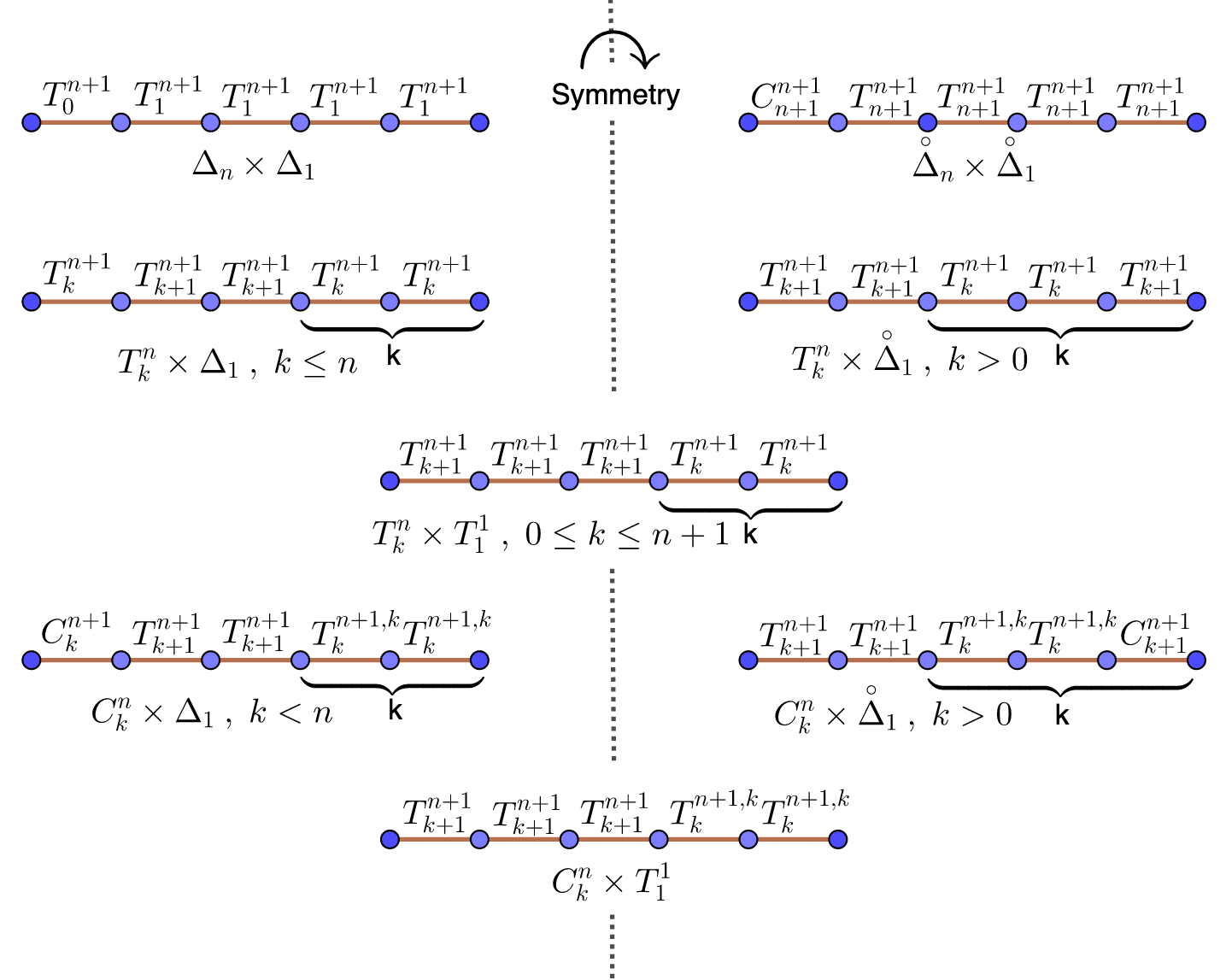}
    \caption{The symmetry observed on some tilings, with $n=4$.}
    \label{figexduality1}
      \end{center}
 \end{figure}

 \begin{figure}[h]
   \begin{center}
    \includegraphics[scale=1]{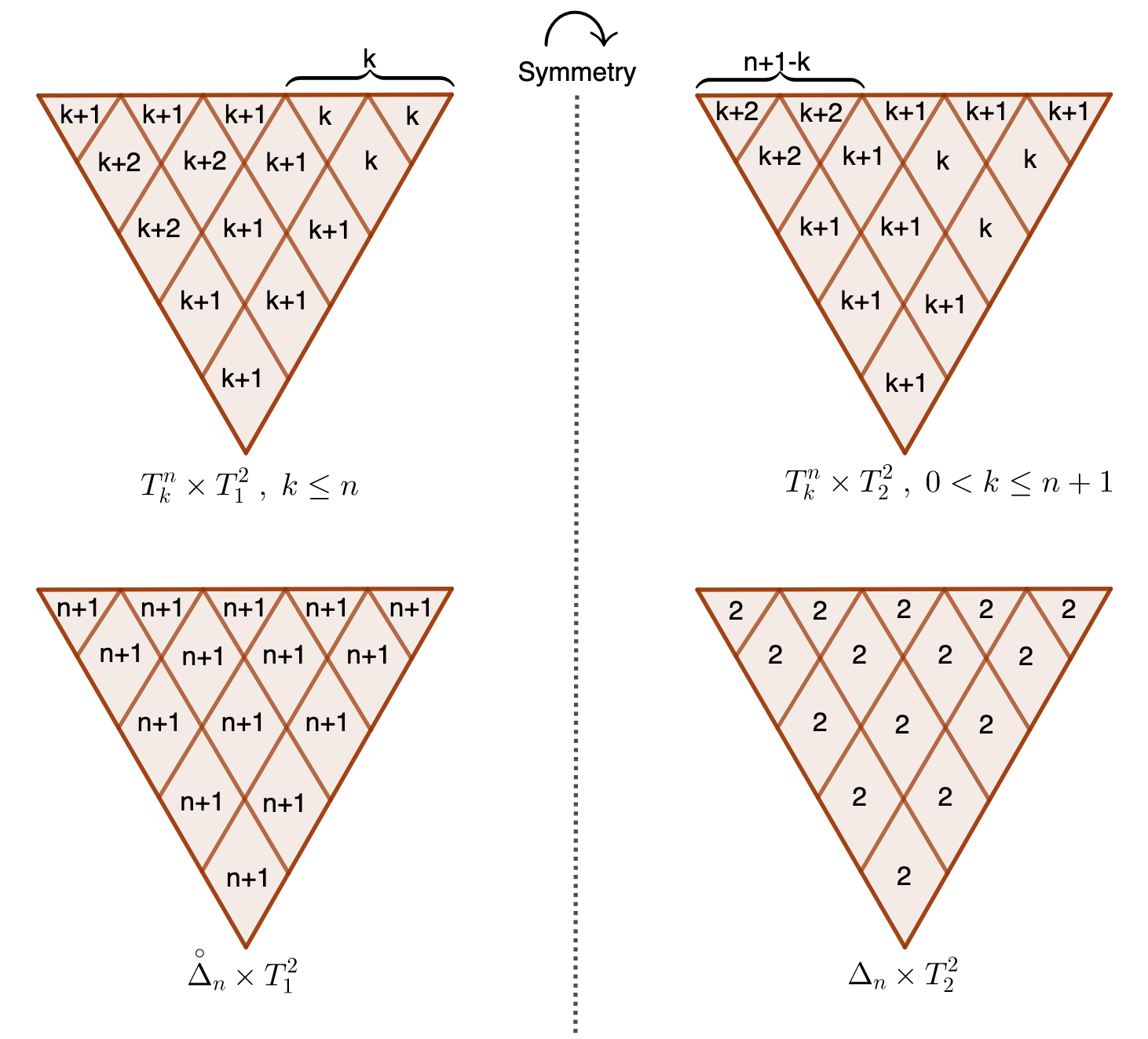}
    \caption{The symmetry observed on more tilings, with $n=4$.}
    \label{figexduality2}
      \end{center}
 \end{figure}

 \begin{figure}[h]
   \begin{center}
    \includegraphics[scale=1]{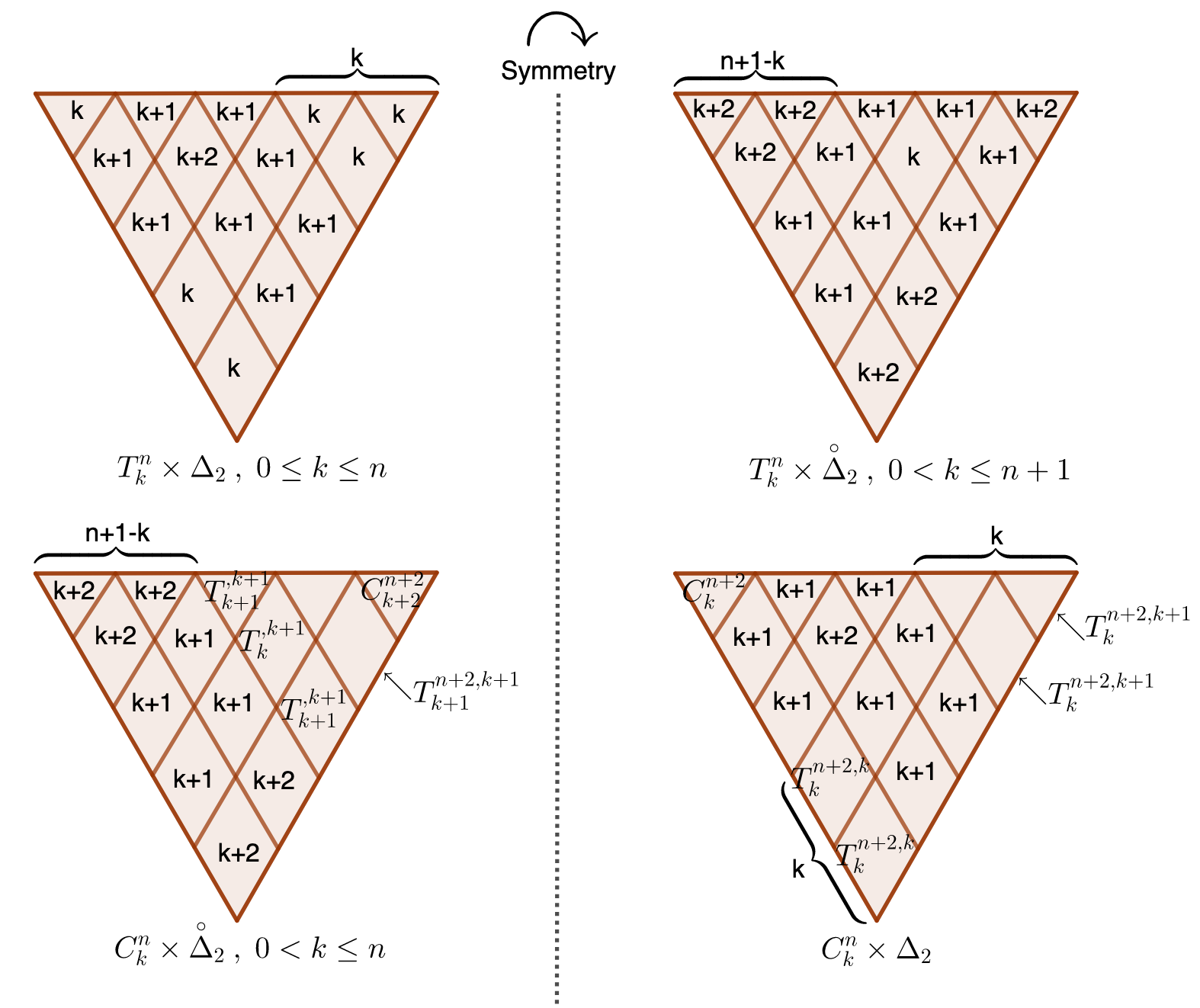}
    \caption{The symmetry observed on more tilings, with $n=4$.}
    \label{figexduality3}
      \end{center}
 \end{figure}

 \begin{figure}[h]
   \begin{center}
    \includegraphics[scale=1]{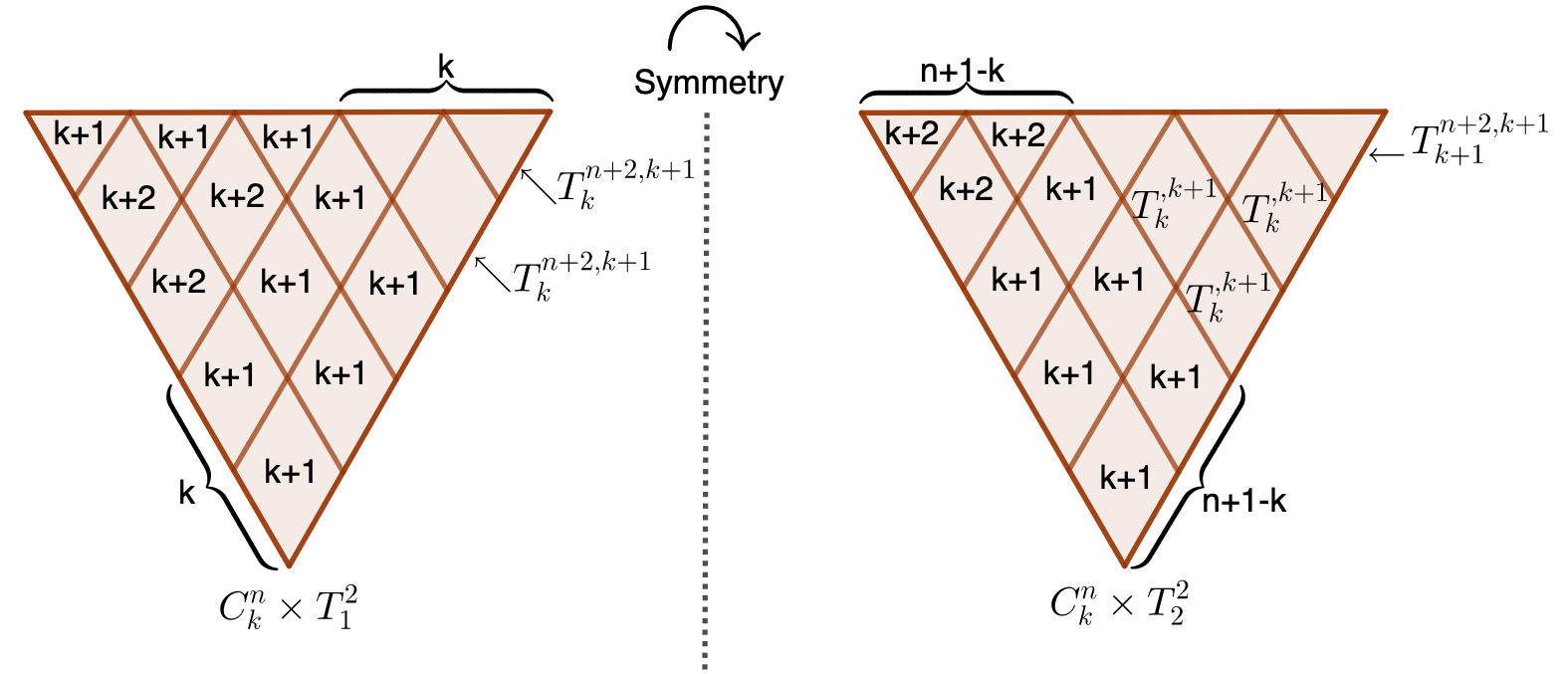}
    \caption{The symmetry observed on more tilings, with $n=4$.}
    \label{figexduality4}
      \end{center}
 \end{figure}

\section{Tilings on products of two complexes}
\label{sectilings}

\subsection{Proof of Theorems \ref{theoproduct} and \ref{theoproducthtiling}}
\label{subsecprooftheoproduct}

Let $K$ and $L$ be the simplicial complexes underlying $S_1$ and $S_2$. Let us equip their edges with orientations given by Definition \ref{deftame}, the tilings of $S_1$ and $S_2$ being tame by hypothesis. Then, the vertices of every simplex $\sigma$ of $K$ and $\theta$ of $L$ inherit a total order given by Proposition \ref{proporder}, so that $\sigma \times \theta$  gets equipped with a staircase triangulation given by Corollary \ref{corprimtriangulations}. Moreover, the face inclusions preserve these orders by Proposition \ref{proporder}, so that the staircase triangulations on these products glue together to define a primitive triangulation on $K \times L$, compare Lemma $II. 8.9$ of \cite{EilSteen}. If $\tau_1$ and $\tau_2$ are shelled, the tiles of $\tau_1$ get labelled $T_1, \dots , T_{N_1}$ and the tiles of $\tau_2$ labelled $T'_1, \dots , T'_{N_2}$. Let us label the underlying simplices $\sigma_1, \dots , \sigma_{N_1}$ and $\theta_1, \dots , \theta_{N_2}$, they shell $K$ and $L$ respectively. The products $\sigma_k \times \theta_l$ get then ordered by the lexicographic order on pairs $(k,l) \in \{1 , \dots , N_1 \} \times \{1, \dots, N_2\}$. By Corollary \ref{corprimtriangulations}, the triangulation of each product $\sigma_k \times \theta_l$ is itself shelled, providing an order on its maximal simplices $(\Delta_I^{(k,l)})_{I \in I(n,m)}$, where $n$ denotes the dimension of $\sigma_k $ and $m$ the dimension of $\theta_l$. The lexicographic order on triplets $(k,l,I) \in \{1 , \dots , N_1 \} \times \{1, \dots, N_2\} \times I(n,m)$ induces then a shelling on the triangulated product $K \times L$. We have here used a slight abuse of notation since the dimension $n$ (resp. $m$) depends on $k$ (resp. $l$) in general. Now, $S_1 \times S_2$ is partitionned by the products $T_k \times T'_l$,
$(k,l) \in \{1 , \dots , N_1 \} \times \{1, \dots, N_2\}$, and by Theorem \ref{theoproducttiles}, these products, equipped with the preceding triangulation, are Morse shellable. Indeed, the tilings of $S_1$ and $S_2$ being tame, we know from Definition \ref{deftame} that Condition $M$ of \S \ref{subseccond} is satisfied and this guaranties the Morse shellability of $T_k \times T'_l$. Again, the lexicographic order on triplets $(k,l,I) \in \{1 , \dots , N_1 \} \times \{1, \dots, N_2\} \times I(n,m)$ induces a Morse shelling on the triangulated product $S_1 \times S_2$. Moreover, this shelling is tame since the ones of $T_k \times T'_l$ are by Theorem \ref{theoproducttiles}.

By Theorem \ref{theoproducttiles}, the critical tiles of $S_1 \times S_2$ are then in bijective correspondence with the products $T_k \times T'_l$ of a critical tile $T_k$ of $S_1$ and a critical tile $T'_l$ of $S_2$, their indices being the sum of the index of $T_k$ with the index of $T'_l$. The $c$-vector of $S_1 \times S_2$  is thus the product $c(S_1)c(S_2)$ by definition. Finally, if the tiles of $S_1$ (resp. of $S_2$) all have same dimension $n$ (resp. $m$) and if $h(S_1)$ and $h(S_2)$ are palindromic, then, we may group the tiles of $S_1$ (resp. of $S_2$) by pairs of tiles of order $j$ and $n+1-j$ (resp. $i$ and $m+1-i$), $0 \leq j < \frac{n+1}{2}$ (resp. 
$0 \leq i < \frac{m+1}{2}$), leaving alone the tiles of order $\frac{n+1}{2}$ (resp. $\frac{m+1}{2}$) in case $n$ (resp. $m$) is odd. The products $T_k \times T'_l$, $(k,l) \in \{1 , \dots , N_1 \} \times \{1, \dots, N_2\}$, are then grouped by quadruples, pairs or left alone depending on the cases, but Theorem \ref{theoduality} ensures that the contribution of each group to the $h$-vector of $S_1 \times S_2$ is palindromic. Adding together the contributions of all these groups, we deduce Theorem \ref{theoproduct}. Now, under the hypothesis of Theorem \ref{theoproducthtiling}, for every
$(k,l) \in \{1 , \dots , N_1 \} \times \{1, \dots, N_2\} $, Condition $h$ of \S \ref{subseccond} gets satisfied by the total orders on the vertices of $T_k$ and $T'_l$, so that the shelling $\big( T_I \cap (T_k \times T'_l)\big)_{I \in I(n,m)}$ given by Theorem \ref{theoproducttiles} uses only basic tiles, which proves Theorem \ref{theoproducthtiling}. $\square$

\subsection{Proof of Theorem \ref{theodelta2}}

The simplicial complex $\partial \Delta_2$ can be equipped with the $h$-tiling using three tiles of dimension and order one, see Example \ref{exhtiling}. We may orient each edge in such a way that it goes towards the remaining vertex of each tile. Let $K$ be the simplicial complex underlying $S$. We choose a total order on its vertices. The conditions of Theorem \ref{theoproducthtiling} are then satisfied so that $S \times \partial \Delta_2$ inherits an $h$-tiled primitive triangulation whose $h$-vector is moreover palindromic provided $h(S)$ is. For every tile $T_1$ of $S$ and $T_2$ of $\partial \Delta_2$, the orders chosen on vertices provide an isomorphism between $T_1 \times T_2$ and $( \Delta_{[n]} \setminus \cup_{j \in J_1} \Delta_{[n] \setminus \{ j \}} ) \times (  \Delta_{\{0, 1\}} \setminus \Delta_{\{0\}} ) $, where $n$ is the dimension of $T_1$. For every $I \in I(n,1)$, only one interval $I_{j_0}$ is not a singleton, $j_0 \in \{0, \dots , n\}$, and the tile $T_I \cap (\Delta_{[n]} \times T_2)$ is of order one. Thus, the tile $T_I \cap (T_1 \times T_2)$ is of order $\# J_1$ if $j_0 \in J_1$ and of order $\# J_1 + 1$ otherwise, see (\ref{eqnarrayorder}). The tiling $\big(T_I \cap (T_1 \times T_2)\big)_{I \in I(n,1)}$ uses then $\ord (T_1)$ basic tiles of order $\ord (T_1)$ and $\dim (T_1) + 1 - \ord (T_1)$ basic tiles of order $\ord (T_1)+1$. Theorem \ref{theodelta2} follows.


Univ Lyon, Universit\'e Claude Bernard Lyon 1, CNRS UMR 5208, Institut Camille Jordan, 43 blvd. du 11 novembre 1918, F-69622 Villeurbanne cedex, France

{\tt welschinger@math.univ-lyon1.fr.}
\end{document}